\RequirePackage[l2tabu, orthodox]{nag}
\documentclass[11pt]{article}%

\usepackage{amssymb}
\usepackage{amsfonts}
\usepackage{amsmath}
\usepackage{mathtools}
\usepackage{dsfont}
\usepackage{mathrsfs}

\usepackage{amsthm}
\usepackage{hyperref}

\usepackage{float}
\usepackage{graphicx}
\usepackage{wrapfig}
\usepackage{subfig}
\usepackage{fancybox}
\usepackage{framed}
\usepackage[usenames,dvipsnames,svgnames,table]{xcolor}
\usepackage{esint}
\usepackage{caption}
\usepackage{epstopdf} 
\usepackage[all]{xy} 
\usepackage{tikz} 
\usepackage{tabularx}
\usepackage{afterpage}
\usepackage{arydshln}

\usepackage{xparse}

\usepackage[titletoc, title]{appendix} 
\usepackage{titling} 
\usepackage{url} 
\usepackage{color}
\usepackage{enumerate, multirow, longtable} 

\usepackage[latin1]{inputenc}
\usepackage[
    a4paper,
    hmargin=2cm,
    top=1.85cm,
    bottom=3.85cm
]{geometry}
\usepackage{microtype}

\newtheorem{theo}{Theorem}[section]

\newtheorem{lem}[theo]{Lemma}
\newtheorem{cor}[theo]{Corollary}

\newtheorem{rem}[theo]{Remark}
\newtheorem{defi}[theo]{Definition}

\numberwithin{equation}{section}

\newcommand{\sts}{\omega} 

\newcommand{\hide}[1]{}

\newcommand{\mb}[1]{\mathbb{ #1 }}
\newcommand{\mbf}[1]{\mathbf{ #1 }}
\newcommand{\mc}[1]{\mathcal{ #1 }}
\newcommand{\mr}[1]{\mathrm{ #1 }}

\newcommand{\dd}{\mathrm{d}}

\newcommand{\ind}{\mathbf{1}}

\newcommand{\eps}{\varepsilon}

\NewDocumentCommand{\Var}{o}{%
  \mathfrak{C}_{\sts\IfValueT{#1}{, #1}}%
}
\NewDocumentCommand{\var}{o}{%
  \mathfrak{c}_{\sts\IfValueT{#1}{, #1}}%
}

\title{Power variations of critical Gaussian multiplicative chaos \\ and their spectral applications}
\author{Mo Dick Wong\thanks{Department of Mathematics, The University of Hong Kong}}
\date{\today}
\begin{document}
\maketitle

\abstract{
We introduce power variations of critical Gaussian multiplicative chaos (GMC) along refining partitions in arbitrary dimension.
Under suitable renormalisation,
we prove uniform fractional moment bounds via Laplace transform estimates as well as stable convergence to supercritical GMCs.
Together, these results provide a unified approach to regularity and spectral questions for critical chaos, 
yielding:
(i) the sharp exponent for the logarithmic modulus of continuity,
(ii) the existence of non-empty essential spectrum for critical Liouville quantum gravity surfaces; and
(iii) an almost-sure quantitative Fourier decay.
}
\section{Introduction}
\subsection{Setup and main results}
Gaussian multiplicative chaos (GMC) is a family of random measures formally given by $e^{\gamma X(x)} \dd x$
where $\gamma > 0$ is an intermittency parameter and $X(\cdot)$ is a log-correlated Gaussian field on some domain $D \subset \mb{R}^d$.
First constructed in Kahane's seminal work \cite{Kah1985},
the theory of GMC was originally motivated by a turbulence model of Kolmogorov-Obukhov-Mandelbrot \cite{Man1971},
but since then has played an indispensable role in random planar geometry and Liouville quantum gravity \cite{DS2011,DKRV2016,BE2026,CKRV2024},
and in recent years it has also found applications in random matrix theory \cite{CFLW2021,KW2022,NPS2023}
as well as connections to number theory \cite{SW2020,Har2020,Har2024}.

The theory of GMC exhibits three distinct regimes,
namely the subcritical ($\gamma \in (0, \sqrt{2d})$), critical ($\gamma = \sqrt{2d}$) and supercritical ($\gamma > \sqrt{2d}$) phases,
each characterised by different renormalisations and markedly different limiting behaviour.
Unlike the subcritical and critical cases in which convergence under suitable renormalisation may be established in probability
(see e.g. \cite{Ber2017,Pow2021,Lac2024}),
the construction of supercritical GMC only holds at the level of stable convergence in distribution
due to the dominant contribution coming from rare extreme values of the underlying field $X(\cdot)$.
Crucially, this resulting limit measure is purely atomic and given by a Poisson random measure whose conditional spatial intensity is proportional to the critical GMC,
a manifestation of the freezing phenomenon in the literature on branching processes and log-correlated fields \cite{MRV2016,BH2025a,BH2025b}.

The purpose of this article is to establish another natural link between critical and supercritical GMCs via power variations.
For simplicity, we restrict ourselves to the case where $X$ is a $\star$-scale invariant Gaussian field on $\mb{R}^d$,
and we denote by $(X_t)_{t \ge 0}$ its martingale decomposition and $(\mc{F}_t)_{t \ge 0}$ the natural filtration (see Definition \ref{def:star-scale-invariant} for details).
The associated chaos at criticality $\gamma = \gamma_c = \sqrt{2d}$ can then be defined via the Seneta-Heyde normalisation
\begin{align*}
M(\dd x) := \lim_{t \to \infty} \sqrt{t} e^{\gamma_c X_t(x) - dt}\dd x
= \lim_{t \to \infty} M_t(\dd x)
\end{align*}

\noindent where the convergence holds in probability with respect to the weak$^*$ topology.
For any $\beta > 1$, $n \in \mb{N}$ and non-negative function $f$,
the $\beta$-variation on the unit cube $Q:=(0,1)^d$ at scale $n$ is defined as
\begin{align}
V_n^{\beta}(f) := \sum_{\mbf{j} \in \mc{J}_{n}} M(Q_{n, \mbf{j}})^{\beta} f(x_{n, \mbf{j}})
\end{align}

\noindent where
\begin{align*}
Q_{\mbf{j}} := \mbf{j}  + (0, 1)^d,
\qquad Q_{n, \mbf{j}} :=  n^{-1}Q_\mbf{j},
\qquad \mc{J}_{n} := \{0, \dots, n-1\}^d
\end{align*}

\noindent and $x_{n, \mbf{j}}$ is the centre of $Q_{n, \mbf{j}}$.
Abusing the notation, we write $V_n^{\beta}(Q)$ when $f|_Q \equiv 1$. 
The shorthand
\begin{align*}
\ell(n) := \log n,
\qquad \widetilde{V}_n^{\beta}(f) := \ell(n)^{\frac{\beta}{2}}V_n^{\beta}(f)
\qquad \text{and} \qquad \widetilde{V}_n^\beta(Q) :=  \ell(n)^{\frac{\beta}{2}}V_n^{\beta}(Q)
\end{align*}

\noindent for renormalised $\beta$-variations is also used.

We study the asymptotic behaviour of $V_n^{\beta}(f)$ as $n \to \infty$ through moment estimates and limiting distributions.
Our first result is the following bound on the Laplace transform:
\begin{theo}\label{thm:uniform-integrability}
For any $\beta > 1$ and $r \in (0, 1/\beta)$, there exists a constant $C = C(\beta, r) \in (0, \infty)$ such that 
\begin{align*}
\sup_{n \ge 1} \mb{E} \left[ 1 - \exp\left\{-q\widetilde{V}_n^\beta(Q)\right\}\right] \le C q^{r}
\qquad \forall q \in [0,1/2].
\end{align*}
\end{theo}

An immediate consequence is the following integrability result.
\begin{cor}\label{cor:uniform-integrability}
For any fixed $\beta > 1$ and $r \in (0, \beta^{-1})$,
we have $\sup_{n \ge 1} \mb{E}\left[\widetilde{V}_n^\beta(Q)^r \right] < \infty$.
\end{cor}

\begin{proof}
Choose $r_0 \in (r, 1/\beta)$. Using the formula
\begin{align}\label{eq:moment-via-laplace}
Y^r
= r\int_0^\infty \ind_{s \le Y} ds \int_0^\infty \frac{q^{-r}e^{-qs}}{\Gamma(1-r)}dq
= \frac{r}{\Gamma(1-r)}\int_0^\infty \left(1- e^{-qY} \right) \frac{dq}{q^{r+1}}
\end{align}

\noindent for any $Y \ge 0$ by Fubini's theorem, we conclude by Theorem \ref{thm:uniform-integrability} that
\begin{align*}
\sup_{n \ge 1} \mb{E}\left[\widetilde{V}_n^\beta(Q)^r\right]
\le \frac{r}{\Gamma(1-r)} \left[\int_0^{1/2} C q^{r_0} \frac{dq}{q^{r+1}} + \int_{1/2}^{\infty} \frac{dq}{q^{r+1}} \right] < \infty.
\end{align*}
\end{proof}

Refining the proof of Theorem \ref{thm:uniform-integrability},
we also establish the following distributional convergence.

\begin{theo}\label{thm:critical-to-supercritical}
Let $\beta > 1$ and $\alpha := 1/\beta$.
Suppose $f$ is a non-negative function of the form $f(x) = f_0(x) \ind_D(x)$,
where $D \subset Q$ is an open Jordan measurable set and $f_0 \in C(\overline{D})$.\footnote{
A set $D \subset \mb{R}^d$ is said to be Jordan measurable if its inner and outer Jordan measures coincide,
or equivalently the Lebesgue measure of its boundary $\partial D$ is equal to $0$.
Intuitively, this means the measure of $D$ is well-approximated by cubes of equal side-length.
Inspecting the proof, 
Theorem \ref{thm:critical-to-supercritical} readily extends to test functions $f$ obtained from positive linear combinations of functions of the form $f_0 \ind_D$.
}
Then for any bounded $\mc{F}_\infty$-measurable random variable $U$, we have
\begin{align}\label{eq:check-stable}
\mb{E}\left[U \exp\left(-\widetilde{V}_n^\beta(f)\right)\right]
\xrightarrow[n \to \infty]{} \mb{E}\left[U \exp\left(-\frac{\Gamma(1-\alpha)}{\sqrt{\pi d}}\int_Q f(x)^{\alpha} M(\dd x)\right)\right].
\end{align}

\noindent Equivalently, if
\begin{align*}
\mc{P}_{\alpha}[M](\dd x) := \int_0^{\infty}z \eta_{\alpha}[M](\dd x, \dd z)
\end{align*}

\noindent where $\eta_{\alpha}[M]$ (conditioned on $\mc{F}_\infty$) is a Poisson random measure on $\mb{R}^d \times (0, \infty)$ with intensity measure $M(dx) \otimes z^{-(1+\alpha)} \dd z$,
then we have 
\begin{align}\label{eq:stable-distributional-convergence}
\widetilde{V}_{n}^{\beta}(f) = \ell(n)^{\frac{\beta}{2}} V_{n}^{\beta}(f) 
\xrightarrow[n \to \infty]{(d)}  \left(\frac{\alpha}{\sqrt{\pi d}}\right)^{\beta} \mc{P}_\alpha[M](f)
\end{align}

\noindent where the distributional convergence is stable with respect to $\mc{F}_\infty$,
the sigma-algebra generated by the martingale decomposition of $X$.
\end{theo}

\begin{rem}
Recall Campbell's theorem for Poisson random measures:
if $\mr{PPP}_{\lambda}(\dd u)$ is a Poisson point process with locally finite intensity measure $\lambda$ on some space $E$
(see e.g. \cite[Lemma 3.1(ii)]{Kal2017} for the setup of abstract Borel space), then
\begin{align*}
\mb{E}\left[ \exp\left(- \int_E \varphi(u)\mr{PPP}_\lambda(\dd u)\right)\right]
= \mb{E}\left[ \exp\left(- \int_E \left(1 - e^{-\varphi(u)}\right)\lambda(\dd u)\right)\right]
\end{align*}

\noindent for suitable test functions $\varphi$.
In particular, by applying the identity \eqref{eq:moment-via-laplace} one can deduce \eqref{eq:stable-distributional-convergence} immediately from \eqref{eq:check-stable}.

Note that $\mc{P}_{\alpha}[M]$ is (up to scalar multiple) the supercritical GMC measure with intermittency parameter $\gamma := \beta \gamma_c$ studied in \cite{MRV2016, BH2025a, BH2025b}.
A distinguishing feature of our result is that the multiplicative constant in the limiting law is completely explicit:
in contrast to the supercritical convergence results in \cite{MRV2016,BH2025b}, 
the dependence on the intermittency parameter and the dimension is given in closed form in \eqref{eq:check-stable} and \eqref{eq:stable-distributional-convergence}.
As we shall see soon, 
this explicit identification follows from the universal tail profile of critical GMCs established in \cite{Won2019}.
\end{rem}

\subsection{Motivation and applications}
Our interest in Theorem \ref{thm:uniform-integrability} and Theorem \ref{thm:critical-to-supercritical} comes from
the need for a better understanding of the multifractal property of critical multiplicative chaos,
which has various important implications in the spectral geometry of Liouville quantum gravity and the harmonic analysis of critical chaos.
We explain these applications below.

\subsubsection{Regularity of critical chaos}
It is well-known that GMC measures in the subcritical phase $\gamma \in (0, \sqrt{2d})$
are supported on $\gamma$-thick points whose Hausdorff dimension is equal to $\left(d - \frac{\gamma^2}{2}\right)$.
This probabilistic observation led to an elegant construction of subcritical GMCs \cite{Ber2017},
and is closely connected to many analytic properties satisfied by subcritical GMCs such as the finiteness of Frostman-type energy and two-sided H\"older estimate;
we also refer readers to \cite{Ber2023} for a nice discussion of related multifractal analysis.

Extending some of these observations to the critical regime $\gamma_c = \sqrt{2d}$,
the relevant set of thick points now has Hausdorff dimension $0$, 
which shows that the critical chaos is supported on a very thin set.
Even though $M$ is non-atomic, 
one expects much poorer regularity properties here compared to the subcritical counterparts.
Indeed, the H\"older regularity in the subcritical regime is lost and replaced by the far weaker logarithmic-type modulus of continuity below:
\begin{theo}\label{thm:log-modulus-of-continuity}
For any $r \in (0, 1/2)$,
there exists an almost surely finite random variable $C = C_r(\omega)$ such that
simultaneously for all cubes $B \subset Q$, we have  
\begin{align}\label{eq:log-modulus-of-continuity}
M(B) \le C \left(\log \left(1 + |B|^{-1}\right)\right)^{-r}
\end{align}

\noindent where $|B|$ denotes the Lebesgue measure of $B$.
\end{theo}

This result was first proved for exactly scale invariant fields in $d \le 2$ in \cite{BKNSW2015}, 
but the same argument can be adapted to general fields 
(under mild decomposition conditions on the covariance kernel and in particular covers $*$-scale invariant fields considered here) 
in arbitrary dimension provided that one has the tail probability estimates obtained in \cite{Won2019} (see Theorem \ref{thm:critical-tail} below).
The optimality of the condition $r < 1/2$ was not established in these previous works,
even though it is highly expected to be the case due to analogies with cascades:
the modulus of continuity of the form \eqref{eq:log-modulus-of-continuity} is true for any $r < 1/2$ but cannot hold for any fixed $r > 1/2$
for critical lognormal cascades (see \cite[Theorem 3]{BKNSW2014}).
We now confirm that the same conclusion holds for critical GMCs.

\begin{theo}\label{thm:optimal-mod-cont}
Let $D \subset \mb{R}^d$ be any non-empty open bounded set. We have 
\begin{align}
\mb{P}\left(\limsup_{|B| \to 0} \sqrt{\log \frac{1}{|B|}} M(B) > 0\right) = 1
\end{align}

\noindent where the lim sup is taken over all cubes $B$ inside $D$.
\end{theo}

Note that Theorem \ref{thm:optimal-mod-cont} is stronger than the analogous result in the cascade setup,
which only ruled out \eqref{eq:log-modulus-of-continuity} for $r$ strictly greater than $\frac{1}{2}$.

\subsubsection{Spectral geometry of critical Liouville quantum gravity}\label{sssec:spectral-lqg}
Our interest in the optimal modulus of continuity is not only motivated by curiosity 
but also its applications to the spectral geometry of Liouville quantum gravity (LQG).

Let $D \subset \mb{R}^2$ be a bounded open domain.
For each $\gamma \in (0, 2)$,
the $\gamma$-LQG surface \cite{DS2011} is a random surface equipped with the formal Riemannian metric tensor
\begin{align*}
e^{\gamma h(x)} \left(\dd x_1^2 + \dd x_2^2\right), \qquad x = (x_1, x_2) \in D
\end{align*}

\noindent where $h(\cdot)$ is a Gaussian free field on $D$ with Dirichlet boundary condition.
In rigorous terms, the random surface may be viewed as a random metric measure space
where the volume form $\mu_{\gamma}(\dd x) = \lim_{\eps \to 0^+} \eps^{\frac{\gamma^2}{2}} e^{\gamma h_\eps(x)} \dd x$
is defined under the general framework of GMC,
and the metric is constructed via Liouville first passage percolation and uniquely characterised by a set of axioms
(see \cite{DDG2023} and the references therein).

In a joint work \cite{BW2023} with Berestycki,
the author initiated the study of LQG surfaces from the perspective of spectral geometry,
i.e. understanding how the geometric properties of LQG surfaces are related to
the spectral properties of the (formal) Laplace-Beltrami operator $\Delta_{\gamma}$.
In terms of the probabilistic language, 
the operator $\frac{1}{2}\Delta_{\gamma}$ may be interpreted as the infinitesimal generator of Liouville Brownian motion \cite{Ber2015,GRV2016} (the canonical diffusion process on LQG surfaces),
and it was established in \cite{MRVZ2016,AK2016} that the non-negative self-adjoint operator $-\tfrac{1}{2}\Delta_{\gamma}$ has compact resolvent 
and hence its discrete spectrum $0 \le \boldsymbol{\lambda}_{\gamma, 1} \le \boldsymbol{\lambda}_{\gamma, 2} \le \dots$ only accumulates at infinity.
One of the main results in \cite{BW2023} was a Weyl law giving the leading-order asymptotic for the eigenvalue counting function 
$\mr{\mbf{N}}_{\gamma}(\lambda) := \#\{\boldsymbol{\lambda}_{\gamma, j} \le \lambda\}$ as $\lambda \to \infty$: we have 
\begin{align}\label{eq:weyl-law}
\frac{1}{\lambda} \mr{\mbf{N}}_{\gamma}(\lambda) \xrightarrow[\lambda \to \infty]{} \frac{\mu_{\gamma}(D)}{\pi(2 - \gamma^2 /2)}
\qquad \text{in probability}.
\end{align}

\noindent This result shows that the spectral dimension of (subcritical) LQG surfaces is equal to $2$ (which was first verified in \cite{RV2014}),
and on top of that the leading order coefficient is described by a new proportionality constant that is distinct from the Riemannian predictions.
A curious observation is that as $\gamma \to 2^-$,
the leading order coefficient on the RHS of \eqref{eq:weyl-law} converges in probability to a multiple of the critical LQG measure \cite{Lac2024},
and it is very natural to ask whether an analogous Weyl law holds at criticality $\gamma_c = 2$.

In order to make sense of such a question,
we must address a much more fundamental problem, 
i.e. whether the operator $\Delta_{\gamma_c}$ still has a discrete spectrum.
We recall that in the subcritical case,
the existence of the discrete spectrum is established by showing
\begin{align}\label{eq:resolvent-HS}
\int_{D \times D} G_D(x, y)^2 \mu_\gamma(\dd x) \mu_{\gamma}(\dd y) < \infty \qquad \text{almost surely},
\end{align}

\noindent which is a sufficient criterion because it shows that the resolvent operator is Hilbert-Schmidt and hence compact.
Since $G_D(x, y) = - \log|x-y| + O(1)$ in the bulk of the domain $D$,
it is not hard to show that an estimate analogous to \eqref{eq:resolvent-HS} does not hold for critical LQG measure.
Indeed Theorem \ref{thm:critical-to-supercritical}  
(and an adaptation of our proof of Theorem \ref{thm:optimal-mod-cont}) would also imply that the Schatten $p$-norm 
of the resolvent operator $\left(-\tfrac{1}{2}\Delta_{\gamma_c} + \lambda\right)^{-1}$ is almost surely infinite for any $p > 1$ and any $\lambda > 0$,
but we will directly establish the following stronger result.
\begin{theo}\label{thm:cLQG-spec}
Almost surely,
the spectrum of the infinitesimal generator of critical Liouville Brownian motion is not discrete,
and the essential spectrum $\sigma_{\mr{ess}}(-\tfrac{1}{2}\Delta_{\gamma_c})$ contains $0$.
\end{theo}

Let us provide a non-rigorous explanation for $0 \in \sigma_{\mr{ess}}(-\tfrac{1}{2}\Delta_{\gamma_c})$ from an analytical perspective.
We recall that there exists a Dirichlet form $(\mc{E}, \mc{D})$ such that
\begin{align*}
  \mc{E}(f, g) := \frac{1}{2} \int_D \nabla f(x) \cdot \nabla g(x) ~\dd x \qquad \text{on} \qquad  C_c^\infty(D) \subset \mc{D}
\end{align*}

\noindent and $\mc{D}$ is the closure of $C_c^\infty(D)$ with respect to the norm
\begin{align}\label{eq:form-norm}
\|f\|_{\mc{D}}^2 :=  \mc{E}(f, f) + \|f\|_{L^2(\mu_\gamma)}^2
\quad \text{where}
\quad \|f\|_{L^2(\mu_\gamma)}^2 = \langle f, f \rangle_{\mu_\gamma} := \int_D f(x)^2\mu_{\gamma}(\dd x).
\end{align} 

\noindent By Kato's representation theorem for bilinear forms (see \cite[Chapter VI.2.1]{Kat1995}),
there exists a unique non-negative self-adjoint operator $\mc{A}$ on $L^2(\mu_{\gamma})$ with $\mr{Dom}(\mc{A}) \subset \mc{D}$ such that  
\begin{align*}
\mc{E}(f, \varphi) = \langle \mc{A}f, \varphi \rangle_{\mu_{\gamma}} \qquad \forall (f, \varphi) \in \mr{Dom}(\mc{A})\times\mc{D}
\end{align*}

\noindent and $-\mc{A}$ is precisely the infinitesimal generator $\tfrac{1}{2}\Delta_{\gamma}$ of Liouville Brownian motion killed upon leaving $D$ \cite{AK2016}.
As such, the eigenvalue equation admits the following weak formulation: 
$(\boldsymbol{\lambda}, \mbf{f}) \in [0, \infty) \times  \mc{D} $ is an eigenpair for the LQG Laplacian $-\tfrac{1}{2}\Delta_{\gamma}$ if
\begin{align*}
\mc{E}(\mbf{f}, \varphi) = \boldsymbol{\lambda} \langle \mbf{f}, \varphi \rangle_{\mu_\gamma}
\qquad \forall \varphi \in C_c^\infty(D).
\end{align*}

Let us denote by $M_{\mr{LQG}}(\dd x)$ the critical LQG measure constructed via Seneta-Heyde renormalisation
so that $(2-\gamma)^{-1} \mu_{\gamma}(\dd x) \to  \sqrt{2\pi} M_{\mr{LQG}}(\dd x)$ as $\gamma \to \gamma_c^- = 2^-$.
Suppose $\gamma$ is very close to the critical threshold, and $(\boldsymbol{\lambda}_{\gamma, j}, \mbf{f}_{\gamma, j})$ is the $j$-th eigenpair of $-\tfrac{1}{2}\Delta_{\gamma}$,
then for any test function $\varphi \in C_c^\infty(D)$ we may expect
\begin{align}\label{eq:heuristic-essential}
\mc{E}(\mbf{f}_{\gamma, j}, \varphi)
= \boldsymbol{\lambda}_{\gamma, j} \langle  \mbf{f}_{\gamma, j}, \varphi \rangle_{\mu_{\gamma}}
\approx \left[ \sqrt{2\pi} (2-\gamma) \boldsymbol{\lambda}_{\gamma, j}\right] \langle  \mbf{f}_{\gamma, j}, \varphi \rangle_{M_{\mr{LQG}}}
\end{align}

\noindent i.e. $ \sqrt{2\pi} (2-\gamma) \boldsymbol{\lambda}_{\gamma, j}$ is an approximate eigenvalue for $-\tfrac{1}{2}\Delta_{\gamma_c}$.
Meanwhile, 
if one extrapolates the Weyl asymptotic uniformly as $\gamma \to 2^-$ down to fixed eigenvalue indices,
the finite limit of its leading-order coefficient suggests that $\boldsymbol{\lambda}_{\gamma,j}$ might remain bounded for each fixed $j$.
This would support the claim that $\lim_{\gamma \to 2^-}  \sqrt{2\pi} (2-\gamma) \boldsymbol{\lambda}_{\gamma, j} = 0$ for every $j$,
and in turn suggest that $0 \in \sigma_{\mr{ess}}(-\tfrac{1}{2}\Delta_{\gamma_c})$.

Of course, the approximation step in \eqref{eq:heuristic-essential} is purely heuristic 
(for instance we do not even know if $\mbf{f}_{\gamma, j} \in L^2(M_{\mr{LQG}})$),
and a priori we do not have an almost-sure finitary upgrade of \eqref{eq:weyl-law} 
that could justify the boundedness of $\boldsymbol{\lambda}_{\gamma, j}$ as $\gamma \to 2^-$.
Our real approach is based on the potential-theoretic framework of Sobolev space by Maz'ya (see Section \ref{sec:sobolev}).
We will show that the embedding $\left(\mc{D}, \|\cdot\|_{\mc{D}}\right) \hookrightarrow L^2(M_{\mr{LQG}})$ is non-compact,
and this is closely related to the failure of a measure Poincar\'e inequality that leads to the vanishing of Rayleigh quotients.

Our result is reminiscent of the sharp criterion in \cite[Theorem 1.4]{HLN2006} 
for self-similar measures generated by iterated function systems satisfying the open set condition:
the measure Poincar\'e inequality and compact Sobolev embedding hold precisely when the lower $L^\infty$-dimension exceeds $d - 2$.
While critical LQG is not covered by this framework,
our logarithmic mass estimates in Corollary \ref{cor:optimal-mod-cont} (combined with the known subcritical theory)
establish an analogous spectral dichotomy for LQG:
the resolvent is compact in the subcritical phase but this no longer holds at criticality.

\subsubsection{Harmonic analysis of critical GMCs}
A different spectral aspect of critical GMCs we are interested in is its Fourier-analytic property.
Given a finite Borel measure $\mu$ on $Q = (0, 1)^d$, define its Fourier coefficients by
\begin{align*}
\widehat{\mu}(\mbf{n})
:= \int_Q e^{2\pi i \mbf{n}\cdot x} \mu(\dd x), \qquad \mbf{n}\in\mb{Z}^d.
\end{align*}

\noindent A fundamental question is how these coefficients behave in the high-frequency limit as $\|\mbf{n}\|_\infty \to \infty$.
In the literature of harmonic analysis,
the measure $\mu$ is called a Rajchman measure if $\widehat{\mu}(\mbf{n})\to 0$ in the limit,
and one defines its Fourier dimension by
\begin{align*}
\mr{dim}_{\mr{F}}(\mu)
:= \sup\left\{s \in [0,d]: |\widehat{\mu}(\mbf{n})|^2 = O(\|\mbf{n}\|_\infty^{-s}) \text{ as }\|\mbf{n}\|_\infty \to \infty\right\}.
\end{align*}

\noindent Our goal here is to investigate these questions for critical GMCs.

The study of Fourier coefficients of GMCs goes back to the work of Falconer and Jin \cite{FJ2019},
who established positive Fourier dimension for subcritical GMCs on suitable planar domains 
with sufficiently small intermittency parameter $\gamma$ through regularity estimates for orthogonal projections.
This subject gained new momentum with the work of Garban and Vargas \cite{GV2026}
which initiated a systematic investigation of the harmonic analysis of GMC on the circle.
Beyond its intrinsic interest,
this problem is motivated by its relation to the Virasoro algebra in Liouville conformal field theory \cite{BGKRV2024};
it also has close parallels in random matrix theory (through the study of holomorphic multiplicative chaos (HMC) \cite{NPS2023})
and analytic number theory (through questions concerning the cancellation phenomenon of partial sums of random multiplicative functions \cite{Har2024}).
For distributional limits, 
we refer the interested reader to \cite{GV2026} for results on real subcritical GMCs,
\cite{NPS2023,NPSV2025,AN2025} for the HMC setup and applications to circular $\beta$-ensembles,
and \cite{GW2024,GW2025a,GW2025b,Har2025,HSX2026} for related recent developments in number theory.

In this paper we are primarily concerned with upper bounds on the modulus of Fourier coefficients.
It is important to point out that the Rajchman property of a measure $\mu$ does not follow immediately from non-atomicity,
let alone any quantitative rate of decay.
That said, 
the prediction of Fourier dimension by Garban-Vargas in $d=1$
was confirmed by Lin, Qiu and Tan in \cite{LQT2024} for subcritical GMC measures on the unit interval,
and analogous results in $d=2$ were obtained by the same authors in \cite{LQT2025} using more sophisticated tools of vector-valued martingales,
which also produce partial results in higher dimensions;
the higher-dimensional cases are then settled by the work of Chen, Lin and Qiu \cite{CLQ2025} 
who showed that subcritical chaos $\mr{GMC}_{\gamma}$ associated to a general class of log-correlated fields on the $d$-dimensional torus satisfy
\begin{align}\label{eq:GMC-subcritical-Fourier-dimension}
\mr{dim}_{\mr{F}}(\mr{GMC}_{\gamma})
= \begin{cases}
d - \gamma^2 & \text{for $\gamma \in (0, \sqrt{2d} / 2)$,}\\
(\sqrt{2d} - \gamma)^2 & \text{for $\gamma \in [\sqrt{2d} / 2, \sqrt{2d})$.}
\end{cases}
\end{align}

\noindent We would also like to highlight the recent work of Orsoni and Verreault \cite{OV2026},
which introduced a different approach based on finite-range decompositions and integration by parts.
Their method yields a substantially simpler proof of the Fourier dimension \eqref{eq:GMC-subcritical-Fourier-dimension} for general $d \ge 1$,
and they also obtain sharp results for restrictions of GMCs to various classes of Euclidean domains, 
highlighting how boundary geometry can reduce the Fourier dimension.

At criticality, the Fourier dimension of GMC is equal to $0$,
which is in agreement with the limiting behaviour of \eqref{eq:GMC-subcritical-Fourier-dimension}.
A priori, it is not clear if the Fourier coefficients $\widehat{M}(\mbf{n})$ would vanish in the high-frequency limit,
and the Rajchman property for critical chaos was posed as an open problem by Garban and Vargas in the circle setting.
The first investigation in this direction was due to Arguin and Hamdan \cite{AH2025},
who considered critical GMCs associated with $\star$-scale invariant field on the unit interval,
and they showed that the $n$-th Fourier coefficient has a decay of $O((\log n)^{{-\frac{1}{4}+\eps}})$ in the sense of convergence in probability.
During the preparation of the present manuscript,
Cai, Chen, Fang and Guo \cite{CCFG2026} established qualitative almost-sure Fourier decay for the canonical critical GMC on the unit circle, 
proving that this measure is almost surely Rajchman.
Our result gives a quantitative affirmative answer to the critical Rajchman question for $\star$-scale invariant fields in arbitrary dimension,
establishing a Fourier decay with logarithmic exponent of $1/2$.

\begin{theo}\label{thm:fourier-decay}
The following statements hold:
\begin{itemize}
\item[(i)] The random variables $\left(\left(\log \|\mbf{n}\|_\infty\right)^{1/2} \widehat{M}(\mbf{n})\right)_{\mbf{n} \ne 0}$ are tight.
\item[(ii)] For any fixed $\eps > 0$, we have $\left(\log \|\mbf{n}\|_\infty\right)^{1/2 - \eps} \widehat{M}(\mbf{n}) 
\to 0$ almost surely as $\|\mbf{n}\|_\infty \to \infty$.
\end{itemize}
\end{theo}

\begin{rem}
Arguin and Hamdan in \cite{AH2025} argued that the decay of $\widehat{M}(\mbf{n})$ might be of order $(\log(\|\mbf{n}\|_\infty))^{-1 +\eps}$,
but following the near-diagonal heuristic underlying their prediction leads us to a different expected scale.
Indeed, their heuristic suggests that $|\widehat{M}(\mbf{n})|^2$ should be comparable in size to $\iint_{|x-y| \le \Delta} M(\dd x) M(\dd y)$ 
where $\log \left(\Delta^{-1}\right) \asymp \log \left(\|\mbf{n}\|_{\infty}\right)$.
The size of this iterated integral should then be comparable to the $2$-variation of $M$ with mesh size of order $\Delta$,
and Theorem \ref{thm:critical-to-supercritical} identifies $\log (\Delta^{-1})$ as the appropriate renormalisation.
We believe that the quantitative rates in our Theorem \ref{thm:fourier-decay} (both for tightness and almost-sure convergence) are essentially optimal.
\end{rem}

\begin{rem}
Theorem \ref{thm:fourier-decay} was stated for $\star$-scale invariant fields in Definition \ref{def:star-scale-invariant} for simplicity,
but our method could be easily adapted to treat more general log-correlated Gaussian fields.
For instance:
\begin{itemize}
\item The smoothness condition on the seed kernel $\rho$ can be significantly weakened,
and one can study the slightly more general class of almost-$\star$-scale invariant fields introduced in \cite{JSW2019};
\item For log-correlated Gaussian fields that admit similar white-noise decompositions as $\star$-scale invariant fields 
(e.g. those arising from the so-called cone constructions), 
our approach could be adapted without any essential changes to the argument.
In particular, these results should hold for the Gaussian free field restricted to the unit circle.
\item For log-correlated Gaussian fields with covariance kernel $-\log|x-y| + g(x, y)$ where $g$ satisfies some minimal Sobolev regularity,
they can be locally decomposed as the sum of an almost-$\star$-scale invariant field and an independent H\"older-continuous Gaussian field (see \cite{JSW2019}).
One would need an estimate slightly different from that in Lemma \ref{lem:control-coarse-field} to take into account of the extra H\"older perturbation, 
but otherwise our technique should extend in a similar way.
\end{itemize}

\noindent Our choice of the domain $Q = (0,1)^d$ is also purely due to convenience,
and one should be able to extend the result to more general setting 
provided that the critical mass of thin boundary layers admits suitable fractional moment estimates, 
as in Corollary \ref{cor:boundary-momentGMC} and Lemma \ref{lem:error-random-shift}.
\end{rem}

\begin{rem}\label{rem:independence}
This work was developed independently of the recent works \cite{CCFG2026,OV2026}.
Although finite-range scale decompositions are a common ingredient,
our mechanism for obtaining Fourier decay is different.
Rather than using pathwise barrier truncation and removing large-mass cells as in \cite{CCFG2026}, 
we work directly with the limiting critical chaos and employ a local randomisation argument.
At a high level, we introduce small independent shifts of the fine chaos $M^{(t)}$ to generate random Fourier phases.
This reduces the Fourier problem to fractional moment bounds for power variations and yields quantitative almost-sure Fourier decay,
thereby linking the Fourier argument to the central estimates of this paper.
We invite interested readers to compare these different approaches to the Fourier analysis of critical GMCs.
\end{rem}

\subsubsection{Discussion}
Although the applications concern different aspects of critical GMC, 
they are linked by two complementary consequences of our power-variation analysis.
The distributional limit forces sufficiently large masses on arbitrarily small sets, 
thereby establishing the optimality of the modulus of continuity and,
through capacity estimates, also the existence of non-empty essential spectrum for critical LQG surfaces.
The uniform fractional moment bounds, on the other hand, control Fourier cancellation through local randomisation.
Power variations thus provide a unified approach to the study of the concentration and cancellation properties here.

Our study of power variations is inspired by \cite[Theorem 5]{BKNSW2014},
where an analogous distributional limit for critical lognormal cascades is established
up to an implicit deterministic normalising sequence bounded above and away from zero.
Their proof uses a generating-function recursion that relies on the exact hierarchical structure of cascades, which is absent in our setting.
Our approach dispenses with any implicit normalising sequence and yields various quantitative upgrades.

A key methodological ingredient in our proof of stable convergence is quantitative Poisson approximation in the presence of finite-range dependence.
The arguments in \cite{MRV2016,BH2025b} for the convergence of supercritical GMCs rely on introducing buffer regions in order to obtain conditional independence,
and the contribution of these buffer regions must be controlled separately.
By contrast,
our analysis of power variations retains the full partition and controls neighbouring interactions directly using mixed-moment bounds for masses on adjacent sets.
It would be interesting to see whether our techniques could be adapted to give a more streamlined proof of convergence of supercritical GMCs.

\subsection{Organisation}
The paper is organised as follows, with the main proof ideas discussed in the respective sections.
\begin{itemize}
\item Section \ref{sec:preliminaries} collects various preliminary results.
We summarise some standard facts about log-correlated Gaussian fields, critical Gaussian multiplicative chaos and Sobolev space theory,
and explain various probabilistic results including a quantitative Poisson approximation theorem.
\item Section \ref{sec:analysis-beta-variation} is dedicated to the proof of Laplace transform estimates and distributional convergence of renormalised power variation,
which relies on the aforementioned Poisson approximation results.
We shall first estimate the error terms ($b_{n, 1}$ and $b_{n,2}$),
and then provide uniform bounds and limiting results for the conditional mean $\lambda_n$ that imply our desired results.

\item In Section \ref{sec:essential-spectrum} we apply these results to establish Theorem \ref{thm:optimal-mod-cont} regarding the optimality of the logarithmic modulus of continuity
and extend it to more general log-correlated Gaussian fields.
This is then used to verify Theorem \ref{thm:cLQG-spec} that asserts that $0$ lies in the essential spectrum of critical LQG surfaces.

\item Finally, Section \ref{sec:Fourier} verifies Theorem \ref{thm:fourier-decay} for the Fourier decay of critical GMCs.
We shall present a key moment estimate which allows us to deduce our quantitative results,
and then explain how this key estimate can be reduced to earlier uniform fractional moment bounds for power variations via a local randomisation trick.
\end{itemize}

\paragraph{Acknowledgements}
The author would like to thank David Croydon and Naotaka Kajino for useful discussions and their hospitality during his visit to 
the Research Institute for Mathematical Sciences at Kyoto University in Fall 2025 when part of this work was done.
He would also like to thank Jad Hamdan for helpful conversations about his work and related heuristics,
and William Verreault for comments on an earlier draft of the manuscript.
This research is supported by a start-up fund from the Faculty of Science, The University of Hong Kong, 
a start-up allowance from the Croucher Foundation, and Hong Kong RGC Grant GRF 17305226.

\section{Preliminaries}\label{sec:preliminaries}
\subsection{Log-correlated Gaussian fields and $\star$-scale invariance}
Let us recall the definition of a $\star$-scale invariant field.
\begin{defi}\label{def:star-scale-invariant}
A centred Gaussian field $X$ on $\mb{R}^d$ is called a $\star$-scale invariant field if its covariance is given by
\begin{align*}
\mb{E}[X(x) X(y)] =: K(x, y) = \int_0^\infty \rho(e^u (x-y)) du = \int_1^\infty \rho(u(x-y))\frac{du}{u}, \qquad \forall x, y \in \mb{R}^d
\end{align*}

\noindent where $\rho: \mb{R}^d \to [0, \infty)$ is a positive definite function satisfying the following properties:
\begin{itemize}
\item[1.] $\rho$ is rotationally symmetric and $\rho(0) = 1$.\footnote{
We occasionally abuse the notation and write $\rho'$ and $\rho''$, having in mind that $\rho$ may be identified as a function on $\mb{R}$ by rotational symmetry.
}
\item[2.] $\rho \in C_c^\infty(\mb{R}^d)$ and $\mr{supp}(\rho) \subset B(0, R/2) \subset \mb{R}^d$ for some $R > 0$.
\end{itemize}

\end{defi}
Let us also recall that a $\star$-scale invariant field $X$ admits a martingale decomposition,
i.e. there exists a family of centred continuous random fields $(X_t(\cdot))_{t \ge 0}$
such that the following are true:
\begin{itemize}
\item For any $0 \le s < t$ and $x, y \in \mb{R}^d$, we have
\begin{align*}
\mb{E}[X_s(x)X_t(y)] = \int_0^{\min(s, t)} \rho(e^u (x-y)) du = \int_1^{e^{\min(s, t)}} \rho(u(x-y)) \frac{\dd u}{u}.
\end{align*}

\noindent Since the covariance kernel is translation invariant, we write $K_t(x-y) := \mb{E}[X_t(x) X_t(y)]$.

\item Denote by $\mc{F}_t:= \sigma\left(X_u(x), u \le t, x \in \mb{R}^d\right)$ the natural filtration.
Then for any $0 \le s < t$, the random field $X_{s, t} := X_t - X_s$ is independent of $\mc{F}_s$.
\item For any fixed $x \in \mb{R}^d$, the process $(X_t(x))_{t \ge 0}$ is a standard Brownian motion.
\item For any $0 \le s < t$, we have $X_{s, t}(\cdot) \overset{(d)}{=} X_{t-s}(e^{s}\cdot)$.
In particular, for any $A, B \subset \mb{R}^d$, 
$(X_{s, t}(x))_{x \in A}$ is independent of $(X_{s, t}(y))_{y \in B}$ as soon as $\inf_{x \in A, y \in B} |x - y|> R e^{-s} / 2$.
\end{itemize}

The collection $(X_t)_{t \ge 0}$ is also called the white-noise decomposition of $X$,
and it is well-known that it satisfies the following upper bound estimates.

\begin{lem}\label{lem:bound-acosta}
For any compact set $K \subset \mb{R}^d$,
there exists some constant $C \in (0, \infty)$ such that for all $t > 0$ sufficiently large, we have 
\begin{align}\label{eq:max-tail-acosta}
\mb{P}\left(\sup_{x \in K} \left|X_t(x)\right| > \left(\sqrt{2d}t - \frac{3/2}{\sqrt{2d}}\log t\right)+ u\right)
\le C e^{-u/C} \qquad \forall u \ge 0.
\end{align}

\noindent In particular, if $a \in (0, \tfrac{3}{2})$,
then 
\begin{align*}
 \mb{P}\left(\sup_{x \in K} |X_t(x)| > \sqrt{2d} t - \frac{a}{\sqrt{2d}} \log t\right) \xrightarrow[t \to \infty]{} 0.
\end{align*}
\end{lem}

\begin{proof}
It is straightforward to check that there exists some constant $C \in (0, \infty)$ independent of $t \ge 1$ such that 
\begin{align*}
\sup_{x, y \in K} \left| K_t(x-y) + \log \max(e^{-t}, |x-y|) \right|  &\le C \\
\text{and} \qquad \mb{E}\left[|X_t(x) - X_t(y)|^2 \right] = 2\left[K_t(0) - K_t(x-y)\right] &\lesssim e^t |x-y| \quad \forall x, y \in K.
\end{align*}

\noindent By covering $K$ with finitely many cubes of unit side-length and applying the union bound,
the estimate \eqref{eq:max-tail-acosta} follows immediately from \cite[Theorem 1.1]{Aco2014}.
In particular, 
\begin{align*}
\mb{P}\left(\sup_{x \in K} |X_t(x)| > \sqrt{2d} t - \frac{a}{\sqrt{2d}} \log t\right)
&\le C\exp\left(-\frac{3/2-a}{C\sqrt{2d}} \log t\right) \xrightarrow[t \to \infty]{} 0,
\end{align*}

\noindent as claimed.
\end{proof}

We will need various estimates to control oscillations of the field $X_t(\cdot)$ at microscopic scale.
Before doing so, let us recall the following classical results in the theory of Gaussian processes.

\begin{lem}\label{lem:Gaussian-analysis}
Let $(\mc{T}, d_{\mc{T}})$ be a metric space, and $\left(\mc{G}_t\right)_{t \in \mc{T}}$ be a separable centred Gaussian process indexed by $\mc{T}$ such that 
\begin{align*}
\mb{P}\left(|\mc{G}_s - \mc{G}_t| \ge u\right) \le 2\exp\left(-\frac{u^2}{2d_{\mc{T}}(s, t)^2}\right) \qquad \forall s, t \in \mc{T}.
\end{align*}

\noindent The following statements hold.
\begin{itemize}
\item For any $u > 0$, we have 
\begin{align}
\label{eq:Borell-TIS}
\mb{P}\left(\sup_{t \in \mc{T}} \mc{G}_t \ge \mb{E} \left[\sup_{t \in \mc{T}} \mc{G}_t\right] + u\right)
& \le \exp\left( - \frac{u^2}{2 \sigma_\mc{T}^2}\right) 
\end{align}

\noindent where $\sigma_{\mc{T}}^2 := \sup_{t \in \mc{T}}\mb{E}[\mc{G}_t^2]$ (see e.g. \cite[Theorem 5.8]{BLG2013}).

\item The expected value of the supremum can be bounded by Dudley's entropy integral (see e.g. \cite[equation (2.41)]{Tal2021}):
there exists some absolute constant $C \in (0, \infty)$ such that
\begin{align}
\label{eq:Dudley-entropy1}
\mb{E}\left[\sup_{t \in \mc{T}} \mc{G}_t\right] & \le C \int_0^\infty \sqrt{\log N(\mc{T}, d_{\mc{T}}, r)}dr
\end{align}

\noindent where $N(\mc{T}, d_{\mc{T}}, r)$ is the covering number of $\mc{T}$, 
i.e. the smallest number $N$ such that the metric space $\mc{T}$ can be covered by $N$ open balls of radius $r$ (with respect to the metric $d_{\mc{T}}$). 
Moreover, for any $\delta > 0$, we have (see e.g. \cite[Theorem 1.42]{Tal2021})
\begin{align}
\label{eq:Dudley-entropy2}
\mb{E}\left[\sup_{s, t \in \mc{T}: d_{\mc{T}}(s, t) \le \delta} |\mc{G}_s - \mc{G}_t|\right] & \le C \int_0^{\delta} \sqrt{\log N(\mc{T}, d_{\mc{T}}, r)}dr
\end{align}
\end{itemize}
\end{lem}

We now record an exponential moment estimate that controls oscillations at microscopic scale.
\begin{lem}\label{lem:lgf-oscillations}
Let $a, L > 0$ be fixed.
There exists some constant $C \in (0, \infty)$ independent of $t > 0$ such that
\begin{align*}
\mb{E}\left[\exp\left(a \sup_{|x| \le Le^{-t}} |X_t(x) - X_t(0)|\right) \bigg| X_t(0) = z\right]
& \le C \exp \left(\frac{C|z|}{t}\right) \qquad \forall z \in \mb{R}.
\end{align*}
\end{lem}

\begin{proof}
Consider 
\begin{align*}
\mb{E}\left[\left(X_t(x) - X_t(0)\right)X_t(0)\right]
= K_t(x) - K_t(0) =: r_t(x).
\end{align*}

\noindent Then we can write 
\begin{align*}
  X_t(x) - X_t(0) = \frac{r_t(x)}{t} X_t(0) + R_t(x)
\end{align*}

\noindent where $\left(R_t(x)\right)_{|x| \le Le^{-t}}$ is independent of $X_t(0)$.
In particular,
\begin{align*}
& \mb{E}\left[\exp\left(a \sup_{|x| \le Le^{-t}} |X_t(x) - X_t(0)|\right) \bigg| X_t(0) = z\right]\\
& \qquad \le \exp\left(\frac{a \sup_{|x| \le L e^{-t}} |r_t(x)|}{t} |z|\right)
\mb{E}\left[\exp\left(a \sup_{|x| \le Le^{-t}} |R_t(x)|\right)\right]
\end{align*}

\noindent where $|r_t(x)| \le \int_0^t |\rho(0) - \rho(e^u x)| du \lesssim  e^t|x| \le L$.
Therefore, we just need to bound the remaining exponential moment involving $R_t(\cdot)$.
For this purpose, note that
\begin{align*}
\mb{E}[R_t(x) R_t(y)] 
&= \mb{E}\left[ \left(X_t(x) - \frac{K_t(x)}{t} X_t(0)\right)\left(X_t(y) - \frac{K_t(y)}{t} X_t(0)\right)\right]\\
& = K_t(x-y) - \frac{K_t(x)K_t(y)}{t}.
\end{align*}

\noindent By the Borell-TIS concentration inequality \eqref{eq:Borell-TIS}, we have 
\begin{align*}
\qquad \mb{P}\left(\sup_{|x| \le L e^{-t}} R_t(x) \ge \mb{E}\left[\sup_{|x| \le L e^{-t}} R_t(x)\right] + u\right) \le e^{-u^2 / (2\sigma^2)} \qquad \forall u \ge 0
\end{align*}

\noindent where 
\begin{align*}
\sigma^2 := \sup_{|x| \le L e^{-t}} \mb{E}[R_t(x)^2]
& = \sup_{|x| \le L e^{-t}} \frac{1}{t} \left[K_t(0)^2 - K_t(x)^2\right]\\
& \le \left( \frac{K_t(0) + \sup_{|x| \le L e^{-t}} K_t(x)}{t}\right)\left(\sup_{|x| \le L e^{-t}} \int_0^t |\rho(0) - \rho(e^u x)|du\right)
\end{align*}

\noindent is uniformly bounded in $t \ge 0$.
Moreover, one can easily check that
\begin{align}\label{eq:natural-distance}
d_t(x, y)^2 := \mb{E} \left[|R_t(x) - R_t(y)|^2\right]
= 2 \left[K_t(0) - K_t(x-y)\right] - \frac{1}{t} \left[K_t(x) - K_t(y)\right]^2
\lesssim e^t|x-y|.
\end{align}

\noindent
This means that the covering number of $\mc{ T}_t:= [-Le^{-t}, Le^{-t}]^d$ satisfies $N(\mc{T}_t, d_t, r) \le \max\left(1, C(L/r^2)^d\right)$ for some absolute constant $C \in (0, \infty)$,
and Dudley's entropy bound \eqref{eq:Dudley-entropy1} shows that
\begin{align*}
  \sup_{t \ge 0} \mb{E}\left[\sup_{|x| \le L e^{-t}} R_t(x)\right]
  \lesssim \sup_{t \ge 0} \int_0^\infty \sqrt{\log N(\mc{T}_t, d_t, r)} dr <\infty.
\end{align*}

\noindent The Gaussian tail in \eqref{eq:Borell-TIS} then implies the uniform boundedness of
\begin{align*}
\mb{E}\left[\exp\left(a \sup_{|x| \le Le^{-t}} |R_t(x)|\right)\right]
& \le \mb{E}\left[\exp\left(a \sup_{|x| \le Le^{-t}} R_t(x)\right)+\exp\left(a \sup_{|x| \le Le^{-t}} \left(-R_t(x)\right)\right)\right]\\
&= 2\mb{E}\left[\exp\left(a \sup_{|x| \le Le^{-t}} R_t(x)\right)\right]
\end{align*}

\noindent for any fixed $a > 0$, and this concludes our proof.
\end{proof}

Last but not least, we record a result concerning the decomposition of log-correlated Gaussian fields.
\begin{theo}[{\cite[Theorem A]{JSW2019}}]\label{thm:field-decomposition}
Let $U \subset \mb{R}^d$ be a domain.
Suppose for $j \in \{1, 2\}$, 
\begin{align*}
C_{j}(x, y) = -\log |x-y| + g_j(x, y) \qquad \forall x, y \in U
\end{align*}

\noindent are non-negative definite kernels with $g_j \in H_{\mr{loc}}^{d+\eps}(U \times U)$ for some $\eps > 0$.\footnote{
We say $f$ belongs to $H_{\mr{loc}}^{s}(U)$ for $s > 0$ if its product with any $\chi \in C_c^\infty(U)$
belongs to the Sobolev space $H^{s}(\mb{R}^d)$ whose norm is defined via Fourier transform in general; 
we refer the interested readers to \cite{JSW2019} for further details.
}
Then for any bounded sub-domain $D$ with $\overline{D} \subset U$,
one can construct (on a suitable probability space) two centred Gaussian fields $G_j$ with covariance kernels $C_j$,
such that  $G_1 - G_2$ is almost surely H\"older-continuous on $\overline{D}$.
\end{theo}

\subsection{Critical Gaussian multiplicative chaos}
Given a $\star$-scale invariant field $X$,
the critical Gaussian multiplicative chaos can be constructed as the weak$^*$ limit of
\begin{align*}
M(\dd x) := \lim_{t \to \infty} \sqrt{t} e^{\gamma_c X_t(x) - dt}\dd x
= \lim_{t \to \infty} M_t(\dd x)
\end{align*}

Recall that $X_{s, t}:= X_{t} - X_s$ for any $0 \le s \le t$.
We can therefore write
\begin{align}
\notag 
M(\dd x) 
& = \lim_{t \to \infty} \sqrt{t+s} e^{\gamma_c X_{t+s}(x) - d(t+s)}\dd x\\
\notag 
&= e^{\gamma_c X_{s}(x) - ds} \lim_{t \to \infty} \sqrt{t} e^{\gamma_c X_{s, s+t}(x) - dt}\dd x\\
\label{eq:coarse-fine-decomposition}
&=: W_s(x) M^{(s)}(\dd x) 
\qquad \text{with}  \qquad W_s(x) = e^{\gamma_c X_{s}(x) - ds}
\end{align}

\noindent where $M^{(s)}(\dd x)$ is independent of $\mc{F}_s$.
Moreover, one has $X_{s, t}(e^{-s}\cdot) \overset{(d)}{=} X_{t-s}(\cdot)$ by inspecting the covariance structure,
and this shows that the measure $M^{(s)}(\dd x)$, after rescaling the spatial variable by $e^{-s}$,
becomes a copy of $M(\dd x)$ up to a factor of $e^{-ds}$ coming from the Jacobian in the change of coordinate.

We now recall the universal tail profile of critical Gaussian multiplicative chaos.
\begin{theo}\label{thm:critical-tail}
Let $\varphi \ge 0$ be a bounded continuous function on $Q$. Then
\begin{align}
\mb{P}\left(\int_Q \varphi(x) M(\dd x) > u\right)
\overset{u \to \infty}{\sim} \frac{\int_Q \varphi(x)\dd x}{u\sqrt{\pi d}}.
\end{align}
\end{theo}

\begin{proof}
This follows immediately from the tail universality result in \cite{Won2019},
since $Q$ is obviously Jordan measurable (i.e. $\partial Q$ has zero Lebesgue measure)
and the decomposition condition for the covariance kernel is satisfied by (regular) $\star$-scale invariant fields
(the readers may refer to the discussions in that paper for the details).
\end{proof}

The following uniform estimate for positive moments under Seneta-Heyde renormalisation is standard, see e.g. \cite[Appendix B.4]{DRSV2014}.
\begin{lem}\label{lem:SH-moment-uniform}
For each $q \in (0, 1)$, we have
\begin{align*}
\sup_{t \ge 0} \mb{E}\left[M_t(Q)^q\right] < \infty.
\end{align*}
\end{lem}

The moment estimate below may also be found in \cite{DRSV2014}.
\begin{lem}\label{lem:moment-criticalGMC}
For each $q \in (0, 1)$, there exists $C = C(q) \in (0, \infty)$ such that 
\begin{align}
\mb{E}\left[M(B(x, r))^q\right] \le C r^{\xi(q)}, \qquad \xi(q) := 2dq - dq^2
\end{align}

\noindent uniformly for all $0 < r \le 1$.
\end{lem}

The small ball result above leads to the following control on the mass of the boundary of the unit cube with respect to critical chaos.
\begin{cor}\label{cor:boundary-momentGMC}
Denote $\partial_h Q := \{x \in [-2, 2]^d: \|x - \partial Q\|_\infty\le h\}$ for any $h > 0$.
For each $q \in (0, 1)$, there exists $C = C(q) \in (0, \infty)$ such that 
\begin{align}
\mb{E}\left[M(\partial_h Q)^q\right]
\le C h^{1 - d(1-q)^2} \qquad \forall h \in [0, 1].
\end{align}
\end{cor}
\begin{proof}
By adjusting the constant $C$,
it suffices to establish the moment estimates for all $h$ sufficiently small.
Cover $\partial_h Q$ with $O(h^{-(d-1)})$ balls of radii $2\sqrt{d} h$ with centres $x \in \partial Q$.
Since $x \mapsto x^q$ is subadditive when $q \in (0, 1)$,
the desired claim follows from Lemma \ref{lem:moment-criticalGMC} with $r = 2\sqrt{d}h$.
\end{proof}

We end this subsection with the following estimate that essentially states that the mass of two neighbouring sets with respect to the same critical chaos cannot be simultaneously huge.
This will be handy since we do not have the convenient exact branching/independence structures seen in the world of multiplicative cascades.
\begin{lem}[{\cite[Lemma 3.2]{Won2019}}]\label{lem:mom-hyperplane}
Let $B_1, B_2$ be bounded disjoint subsets separated by a hyperplane, 
i.e. there exists some $v \in \mb{R}^d \setminus \{0\}$ and $c \in \mb{R}$ such that 
\begin{align*}
B_1 \subset \{x \in \mb{R}^d: \langle x, v\rangle \le c\}
\qquad \text{and} \qquad 
B_2 \subset \{x \in \mb{R}^d: \langle x, v\rangle \ge c\}.
\end{align*}

\noindent Then for any $h \in [\tfrac{1}{2}, \tfrac{1}{2} + \tfrac{1}{2\sqrt{d}})$, we have 
\begin{align*}
\mb{E}\left[ \left(M(B_1) M(B_2)\right)^h\right] < \infty.
\end{align*}
\end{lem}

\subsection{Sobolev space}\label{sec:sobolev}
In this subsection we collect a few facts about Sobolev space.
For any $\mbf{n} \in \mb{N}_0^d$,
we use the standard multi-index notation for mixed partial derivatives
$\partial^{\mbf{n}} := \partial_{1}^{n_1} \cdots \partial_{d}^{n_d}$
and we say $\partial^{\mbf{n}}$ is of order $\|\mbf{n}\|_1 := \sum_{i=1}^d n_i$.
For any bounded open domain $U \subset \mb{R}^d$, $k \in \mb{N}$ and $p \ge 1$,
the Sobolev space $W^{k, p}(U)$ consists of all functions $\varphi: U \to \mb{R}$
such that $\varphi$ and all of its mixed weak derivatives of order $\le k$ are in $L^p(U)$,
and it is a Banach space when equipped with the norm
\begin{align*}
\|\varphi\|_{W^{k,p}(U)} := \left(\sum_{\|\mbf{n}\|_1 \le k} \|\partial^{\mbf{n}} \varphi\|_{L^p(U)}^p\right)^{1/p}.
\end{align*}

\noindent Also, for any $\kappa \in (0, 1)$, 
the $\kappa$-H\"older norm of a function $\varphi:U \to \mb{R}$ is defined by
\begin{align*}
\|\varphi\|_{C^\kappa(U)} := \sup_{x \in U} |\varphi(x)| + [\varphi]_{C^{\kappa}(U)}
\quad \text{with} \quad 
[\varphi]_{C^{\kappa}(U)} := \sup_{\substack{x, y \in U \\ x \ne y}} \frac{|\varphi(x) - \varphi(y)|}{|x-y|^\kappa},
\end{align*}

\noindent and $C^{\kappa}(U) := \{\varphi \in C(U): \|\varphi\|_{C^{\kappa}(U)} < \infty\}$.\\

We now record the following version of Sobolev inequality:
\begin{theo}\label{thm:sobolev-inequality}
Let $Q = (0, 1)^d$ and $p_d := 2(d+1)$.
Then there exists some constant $C \in (0, \infty)$ such that 
for any $x_0 \in \overline{Q}$, we have 
\begin{align*}
[\varphi]_{C^{1/2}(Q)} \le \|\varphi - \varphi(x_0)\|_{C^{1/2}(Q) } \le C \|\nabla \varphi\|_{L^{p_d}(Q)} \qquad \forall \varphi \in W^{1, p_d}(Q).
\end{align*}
\end{theo}

\begin{proof}
Since $\tfrac{1}{2} < 1 - d/p_d$, the Sobolev inequality \cite[Theorem 6 in Section 5.6.3]{Eva2010} says that\footnote{
In the book this is stated for domains with $C^1$ boundary because it relies on an extension theorem in \cite[Section 5.4]{Eva2010},
but one can easily check that the proof of the latter result works for cubes in $\mb{R}^d$.
Alternatively, a general Sobolev inequality for bounded domains having the so-called cone property can be found in e.g. \cite[Section 1.1.11]{Maz2011}.
}
\begin{align}\label{eq:sobolev-inequality}
\|\widetilde{\varphi}\|_{C^{1/2}(Q) } 
\lesssim \|\widetilde{\varphi}\|_{W^{1, p_d}(Q)} 
\lesssim \|\widetilde{\varphi}\|_{L^{p_d}(Q)} + \|\nabla\widetilde{\varphi}\|_{L^{p_d}(Q)}
\qquad \forall \widetilde{\varphi} \in W^{1, p_d}(Q).
\end{align}

\noindent Moreover,
the Poincar\'e-Wirtinger inequality \cite[Theorem 1 in Section 5.8.1]{Eva2010} states that\footnote{
Again the $C^1$ boundary is unnecessary and the result applies to cubes.} 
\begin{align}\label{eq:poincare-wirtinger}
  \|\widetilde{\varphi}\|_{L^{p_d}(Q)} \lesssim \|\nabla \widetilde{\varphi}\|_{L^{p_d}(Q)}
\end{align}
\noindent for any $\widetilde{\varphi} \in W^{1, p_d}(Q)$ satisfying $\int_Q \widetilde{\varphi}(x) dx = 0$.

Now take any $\varphi \in W^{1, p_d}(Q)$, and define $\widetilde{\varphi} := \varphi - \varphi_Q$ 
where $\varphi_Q := \int_Q \varphi(x) dx$ the average value of $\varphi$ on the unit cube $Q$.
We have 
\begin{align*}
\|\varphi - \varphi(x_0)\|_{C^{1/2}(Q)}
&= \sup_{x \in Q}|\varphi(x) - \varphi(x_0)| + [\varphi - \varphi(x_0)]_{C^{1/2}(Q)}\\
&= \sup_{x \in Q}|\widetilde{\varphi}(x) - \widetilde{\varphi}(x_0)| + [\widetilde{\varphi} - \widetilde{\varphi}(x_0)]_{C^{1/2}(Q)}
\le 2\|\widetilde{\varphi} \|_{C^{1/2}(Q)}.
\end{align*}

\noindent Our claim then follows by combining \eqref{eq:sobolev-inequality} with \eqref{eq:poincare-wirtinger}.
\end{proof}

When $p = 2$, the Sobolev space $H^1(U) := W^{1, 2}(U)$ is a Hilbert space.
We will also need a variant of Sobolev space with Dirichlet boundary condition,
and we denote by $H_0^1(U)$ the completion of $C_c^\infty(U)$ with respect to the norm $\|\cdot \|_{H^1(U)}$.
For any compact set $K \subset U$, we define its relative $H_0^1(U)$-capacity (or simply capacity) by
\begin{align}
\mr{Cap}(K; U) := \inf\left\{\frac{1}{2} \int_U |\nabla u(x)|^2 \dd x: u \in C_c^\infty(U), u|_{K} \ge 1\right\}.
\end{align}

The result below is a special case of \cite[Theorem 2.3.3 and Corollary 2.3.4]{Maz2011}.
\begin{theo}\label{thm:equivalence-poincare-capacity}
Let $\mu$ be a finite Borel measure on $U$. The following statements are equivalent.
\begin{itemize}
\item[(1)] There exists some constant $C_1 \in (0, \infty)$ such that 
\begin{align}
\int_U u^2 d\mu \le C_1 \int_U |\nabla u|^2 dx \qquad \forall u \in C_c^\infty(U).
\end{align}

\item[(2)] There exists some constant $C_2\in (0, \infty)$ such that 
\begin{align}
\mu(K) \le C_2 \mr{Cap}(K; U) \qquad \forall \text{ compact $K \subset U$}.
\end{align}
\end{itemize}
\end{theo}

In order to apply Theorem \ref{thm:equivalence-poincare-capacity},
we need the following capacity estimate.
\begin{lem}\label{lem:key-capacity-estimate}
Let $D_1 \subset D_2 \subset D_3$ be non-empty bounded open subsets of $\mb{R}^2$ satisfying $\overline{D_1} \subset D_2$.
There exist some constant $C \in (0, \infty)$ and $r_* > 0$ such that the following holds:
for any open ball $B$ with radius $r \le r_*$ such that $\overline{B} \subset D_1$,
there exists $\eta_B \in C_c^\infty(D_2)$ such that
\begin{align}\label{eq:eta-desired-support}
0 \le \eta_B \le 1, \qquad \eta_B |_{\overline{B}} \equiv 1
\end{align}

\noindent and 
\begin{align}\label{eq:key-capacity-estimate}
\frac{1}{2} \int_{D_2} |\nabla \eta_B|^2 dx \le \frac{C}{\log \left(1+|B|^{-1}\right)}
\end{align}

\noindent where $|B|$ is the Lebesgue measure of $B$. In particular,
\begin{align}
\mr{Cap}(\overline{B}; D_3) \le \mr{Cap}(\overline{B}; D_2)  \le \frac{C}{\log\left(1 + |B|^{-1}\right)}.
\end{align}

\end{lem}

\begin{proof}
This result is standard but we provide a proof here for completeness.

Let $\delta := d(D_1, D_2^c)$, $R := \delta / 2$ and $r_* := \delta / 100$.
Suppose $B = B(x_0, r)$ is an open ball centred at $x_0$ with radius $r \le r_*$ and $\overline{B(x_0, r)} \subset D_1$.
Define the function 
\begin{align*}
f(x) := \begin{dcases}
1 & |x-x_0| \le 2r, \\
\frac{\log(R / |x-x_0|)}{\log(R/(2r))} & |x-x_0| \in [2r, R], \\
0 & |x-x_0| \ge R.
\end{dcases}
\end{align*}

\noindent Then it is straightforward to check that 
\begin{align}\label{eq:capacity-concentric}
\frac{1}{2} \int_{D_2} |\nabla f|^2 dx
=  \frac{1}{2}\int_{2r \le |x-x_0|\le R} \frac{\dd x}{\left[|x - x_0| \log(R / (2r))\right]^2}
= \frac{\pi}{\log^2(R/(2r))} \int_{2r}^R \frac{du}{u} = \frac{\pi}{\log(R/(2r))}.
\end{align}

Now let $\varphi \in C_c^\infty(\mb{R}^2)$ be a non-negative mollifier with compact support in $B(0, r)$ and set $\eta_B := f \ast \varphi$.
One can verify that $\eta_B$ satisfies the desired conditions \eqref{eq:eta-desired-support}.
Moreover
\begin{align*}
\frac{1}{2} \int_{D_2} |\nabla \eta_B|^2 \dd x 
= \frac{1}{2} \int_{D_2} \left| \nabla \left(f \ast \varphi\right)(x)\right|^2 \dd x
&= \frac{1}{2} \int_{D_2} \left|\int_{\mb{R}^2} \left(\nabla f(x-y)\right) \varphi(y)\dd y\right|^2 \dd x\\
& \le \int_{\mb{R}^2}\left[\frac{1}{2} \int_{D_2} \left| \left(\nabla f(x-y)\right) \right|^2 \dd x \right] \varphi(y)\dd y
= \frac{\pi}{\log(R/(2r))}
\end{align*}

\noindent where the second last step follows from Jensen's inequality and the last step from \eqref{eq:capacity-concentric}.
Since $R$ is fixed, 
there exists some constant $C \in (0, \infty)$ independent of $B$ 
such that the previous inequality can be rewritten in the form of \eqref{eq:key-capacity-estimate}.
This concludes our proof.
\end{proof}

\subsection{Miscellaneous lemmas}
We collect a few more probability results.
The first one is a useful moment estimate.
\begin{lem}\label{lem:high-moment-estimate}
Let $\beta \in (1, 2]$,
and $Y_j$ be a finite collection of independent centred complex-valued random variables with deterministic bounds $|Y_j| \le m_j$.
There exists some constant $C = C(\beta) \in (0, \infty)$ such that
\begin{align*}
\forall r \ge 2, 
\qquad  \mb{E}\left[\bigg|\sum_j Y_j\bigg|^r\right]^{1/r}
\le C r^{1-1/\beta} \left(\sum_{j} m_j^\beta\right)^{1/\beta}.
\end{align*}
\end{lem}

\begin{proof}
By separately considering the real and imaginary parts of $Y_j$,
we may assume without loss of generality that our random variables are real-valued.

Writing $V := \sum_{j} m_j^\beta$, we consider
$\mb{E}\left[\big|\sum_j Y_j \big|^r\right]^{1/r} \le \mb{E}\left[|S_+|^r\right]^{1/r} + \mb{E}\left[|S_-|^r\right]^{1/r}$ where 
\begin{align*}
S_+ := \sum_j Y_j \ind_{\{m_j > (V/r)^{1/\beta} \}}
\qquad  \text{and} \qquad S_- := \sum_j Y_j \ind_{\{m_j \le  (V/r)^{1/\beta} \}}.
\end{align*}

The first term can be easily bounded by
\begin{align*}
\mb{E}\left[|S_+|^r\right]^{1/r}
\le \sum_j m_j \ind_{\{m_j > (V/r)^{1/\beta} \}} 
\le \left( (V/r)^{1/\beta}\right)^{1-\beta} \sum_j m_j^\beta
\le r^{1-1/\beta} V^{1/\beta}.
\end{align*}

As for the second term, we recall from Hoeffding's inequality (see e.g. \cite[Theorem 2.8]{BLG2013}) that
\begin{align*}
\mb{P}\left(|S_-| > t\right) 
\le 2 \exp\left(-\frac{t^2}{2 \sum_{j} m_j^2\ind_{\{m_j \le  (V/r)^{1/\beta}\}}  }\right) \qquad \forall t \ge 0.
\end{align*}

\noindent Using the identity $\mb{E}\left[|S_-|^r\right] = r \int_0^\infty t^{r-1}\mb{P}(|S_-| > t) dt$, we have 
\begin{align*}
\mb{E}\left[|S_-|^r\right]^{1/r}
& \le \left[2r \int_0^{\infty} t^{r-1} \exp\left(-\frac{t^2}{2 \sum_{j} m_j^2 \ind_{\{m_j \le (V/r)^{1/\beta} \}}}\right)dt\right]^{1/r}\\
&= \left[2\Gamma\left(1 + \frac{r}{2}\right)\right]^{1/r} \left(2 \sum_{j} m_j^2 \ind_{\{m_j \le  (V/r)^{1/\beta} \}}\right)^{1/2}.
\end{align*}

\noindent Since $\Gamma(1 + x) \sim \sqrt{2\pi x} (x/e)^x$ as $x \to \infty$ by Stirling's approximation,
we see that $\left[2\Gamma\left(1 + \frac{r}{2}\right)\right]^{1/r} \lesssim \sqrt{r}$.
In particular,
\begin{align*}
\mb{E}\left[|S_-|^r\right]^{1/r}
\lesssim \left(r\sum_{j} m_j^2 \ind_{\{m_j \le  (V/r)^{1/\beta} \}}\right)^{1/2}
\le \left[r \left((V/r)^{1/\beta}\right)^{2-\beta}\sum_{j} m_j^\beta\right]^{1/2}
= r^{1 - 1/\beta} V^{1/\beta}
\end{align*}

\noindent and our proof is complete.
\end{proof}

In our study of the Fourier decay,
we will use the following corollary of Lemma \ref{lem:high-moment-estimate}.
\begin{cor}\label{cor:high-moment-estimate}
Let $\Lambda \subset \mb{Z}^d$ be a non-empty finite set.
Suppose $U_j $ is a sequence of independent random vectors in $\mb{R}^d$ and $\xi_j(\mbf{n}) := e^{2\pi i \mbf{n} \cdot U_j} - \mb{E}[e^{2\pi i \mbf{n} \cdot U_j}]$.
Then for each $p \in (0, 1)$ and $\beta \in (1, 2]$,
there exists some constant $C = C(p, \beta) \in (0, \infty)$ such that 
\begin{align*}
\mb{E}\left[ \max_{\mbf{n} \in \Lambda} \left| \sum_j \xi_j(\mbf{n}) \mc{A}_j(\mbf{n}) \right|^p\right]
\le C \ell(e|\Lambda|)^{p(1-1/\beta)} \left(\sum_{j} m_j^{\beta}\right)^{p/\beta}
\end{align*}

\noindent uniformly for all deterministic $\mc{A}_j: \Lambda \to \mb{C}$ satisfying $\max_{\mbf{n}} |\mc{A}_j(\mbf{n})| \le m_j$ for each $j$.
\end{cor}

\begin{proof}
Let $r \ge 2$. Since $p, 1/\beta \in (0, 1)$, we see that
\begin{align*}
\mb{E}&\left[ \max_{\mbf{n} \in \Lambda} \left| \sum_j \xi_j(\mbf{n}) \mc{A}_j(\mbf{n}) \right|^p\right]
\le \mb{E}\left[ \left(\sum_{\mbf{n} \in \Lambda} \left| \sum_j \xi_j(\mbf{n}) \mc{A}_j(\mbf{n}) \right|^r\right)^{p/r}\right]\\
&\qquad \le \left( \sum_{\mbf{n} \in \Lambda} \mb{E}\left[ \left| \sum_j \xi_j(\mbf{n}) \mc{A}_j(\mbf{n}) \right|^r\right]\right)^{p/r}
\le C |\Lambda|^{p/r} r^{p(1-1/\beta)} \left(\sum_j m_j^{\beta}\right)^{p/\beta}
\end{align*}

\noindent where the second step follows from H\"older's inequality and the last step is obtained from Lemma \ref{lem:high-moment-estimate}.
Since the constant $C$ in the last bound is independent of $r$,
we can choose $r := 2  \ell(e|\Lambda|)$ so that $|\Lambda|^{1/r} \le 2$ and the moment estimate follows immediately.
\end{proof}

The next result translates the asymptotics of tail probabilities to those of Laplace transforms.
\begin{lem}\label{lem:laplace-estimate}
Let $C \in (0, \infty)$ and $Y$ be a non-negative random variable with the tail asymptotics $\mb{P}(Y > y) \sim C / y$ as $y \to \infty$.
As $q \to 0^+$, we have 
\begin{align*}
\mb{E}\left[1 - e^{- q Y^\beta}\right]
 \sim C \Gamma(1 - \alpha) q^\alpha
\quad \text{and} \quad \mb{E}\left[\min(1, qY^\beta)\right] \sim \frac{C q^\alpha}{1-\alpha}
\end{align*}

\noindent for any fixed $\beta > 1$ and $\alpha := 1/\beta$.\\

If instead $Y$ only satisfies the one-sided bound 
\begin{align*}
  \mb{P}(Y > y) \le \frac{C}{y} \qquad \left(\text{resp.} \quad \mb{P}(Y > y) \ge \frac{C}{\max(1, y)}\right)
\end{align*}
\noindent for all $y > 0$, then
\begin{align*}
\textstyle
\mb{E}\left[1 - e^{- q Y^\beta}\right]\le C \Gamma(1 - \alpha) q^\alpha
\quad &\text{and}\quad \mb{E}\left[\min(1, qY^\beta)\right] \le \frac{C q^\alpha}{1-\alpha}  \quad \forall q \ge 0\\
\bigg(\text{resp.}~\mb{E}\left[1 - e^{- q Y^\beta}\right] \ge \frac{C(1 - 2^{-(1-\alpha)})}{e(1-\alpha)} q^\alpha
\quad&\text{and}\quad\mb{E}\left[\min(1, qY^\beta)\right] \ge \frac{C(1 - 2^{-(1-\alpha)})}{(1-\alpha)} q^\alpha 
\quad \forall q \in [0, \tfrac{1}{2}].\bigg)
\end{align*}
\end{lem}

\begin{proof}
Applying Fubini's theorem, we have 
\begin{align*}
\mb{E}\left[1 - \exp\left(- q Y^\beta\right)\right]
= \int_0^{\infty} e^{-s} ds\mb{E}\left[\ind_{s \le qY^{\beta}}\right]
&= \int_0^{\infty} e^{-s} \mb{P}(Y \ge (s/q)^\alpha)ds\\
& \overset{q \to 0^+}{\sim} \int_0^{\infty} e^{-s} \frac{C}{(s/q)^{\alpha}}ds 
= C \Gamma(1-\alpha) q^{\alpha}
\end{align*}

\noindent which is our first asymptotic formula.
As for the second estimate, we have 
\begin{align*}
\mb{E}\left[\min(1, qY^\beta)\right] 
= \int_0^{1} \mb{P}(qY^{\beta} > u) du
&= \int_0^{1} \mb{P}\left(Y > (u/q)^\alpha \right) du\\
&\overset{q \to 0^+}{\sim}  \int_0^{1} \frac{C}{(u/q)^\alpha} du
=\frac{Cq^{\alpha}}{1-\alpha}.
\end{align*}

The one-sided bounds follow from similar derivations which we omit here.
\end{proof}

Now let us record a moment bound for the $L^p$-norm of a centred Gaussian field based on its covariance.
\begin{lem}\label{lem:lp-bound-gaussian}
Let $Z(\cdot)$ be a centred Gaussian field on a bounded open domain $D \subset \mb{R}^d$
with a covariance kernel $\mb{E}[Z(x) Z(y)] =: K_Z(x - y)$ that is invariant under translation and rotation.
If $K_Z \in C_c^2(\mb{R}^d)$, 
then for any $p, m \ge 1$ we have 
\begin{align*}
\mb{E} \left[\|\nabla Z\|_{L^{p}(D)}^m\right]^{1/m}
\le C d|D|^{1/p}  \sqrt{\max(p, m)} \sqrt{-\partial^2_1 K_Z(0)}
\end{align*}

\noindent where $C \in (0, \infty)$ is some absolute constant.
\end{lem}

\begin{proof}
By the rotational invariance of the covariance kernel $K_Z$ (and hence the field $Z$),
we have
\begin{align}\label{eq:lpbound-gaussian}
\mb{E} \left[\|\nabla Z\|_{L^{p}(D)}^m\right]^{1/m}
\le d \mb{E} \left[\left( \int_D |\partial_1 Z(x)|^p dx\right)^{m/p}\right]^{1/m}.
\end{align}

\noindent If $m \le p$, we may apply H\"older's inequality which shows that \eqref{eq:lpbound-gaussian} is bounded by
\begin{align*}
d \left( \int_D \mb{E} \left[ |\partial_1 Z(x)|^p \right] dx\right)^{\tfrac{m}{p} \cdot \tfrac{1}{m}}
= d |D|^{\frac{1}{p}} \sqrt{\mr{Var}\left(\partial_1 Z(0)\right)} \mb{E}[|G|^p]^{1/p}
\end{align*}

\noindent where $G$ is a standard Gaussian random variable.
As for the complementary case $m \ge p$, one can apply Minkowski's inequality which upper bounds \eqref{eq:lpbound-gaussian} by
\begin{align*}
d \left( \int_D \mb{E} \left[ |\partial_1 Z(x)|^{p \cdot \tfrac{m}{p}} \right]^{\tfrac{p}{m}} dx\right)^{\tfrac{1}{p}}
= d |D|^{\frac{1}{p}} \sqrt{\mr{Var}\left(\partial_1 Z(0)\right)} \mb{E}[|G|^m]^{1/m}.
\end{align*}

\noindent The proof is now complete by noting that $\mb{E}[|G|^p]^{1/p} \lesssim \sqrt{p}$ and $\mr{Var}\left(\partial_1 Z(0)\right) = -\partial_1^2 K_Z(0)$.
\end{proof}

Finally, we record a classical result that enables one to perform Poisson approximation of sum of random variables with finite-range dependence.
\begin{theo}\label{thm:poisson-approximation}
Let $\mc{J}$ be a finite index set,
and $(Y_j)_{j \in \mc{J}}$ be a collection of $[0,1]$-valued random variables with dependency neighbourhoods $(\mc{B}_j)_{j \in \mc{J}}$, 
i.e. the joint distribution of $(Y_k)_{k \not \in \mc{B}_j}$ is independent of $Y_j$ for each $j \in \mc{J}$.
Writing
\begin{align}
\lambda := \sum_{j \in \mc{J}} \mb{E}[Y_j],
\qquad b_1 := \sum_{j \in \mc{J}} \sum_{k \in \mc{B}_j} \mb{E}[Y_j]\mb{E}[Y_k],
\qquad b_2 := \sum_{j \in \mc{J}} \sum_{j \ne k \in \mc{B}_j} \mb{E}[Y_j Y_k],
\end{align}

\noindent we have
\begin{align}
\left|\mb{E}\left[ \prod_{j \in \mc{J}} (1 - Y_j)\right] - e^{-\lambda}\right|
\le \min(1, \lambda^{-1}) (b_1 + b_2).
\end{align}
\end{theo}

\begin{proof}
Let $U_j$ be i.i.d. $\mr{Uniform}([0,1])$ random variables independent of $(Y_j)_{j \in \mc{J}}$.
Then $W := \sum_{j \in \mc{J}} \ind_{\{U_j \le Y_j\}}$ satisfies
\begin{align*}
\mb{P}(W = 0) = \mb{E}\left[ \mb{P}\left(\bigcap_{j\in\mc{J}} \{U_j > Y_j\} \bigg| (Y_j)_j\right)\right]
= \mb{E}\left[ \prod_{j\in \mc{J}} (1 - Y_j)\right]
\end{align*}

\noindent and our claim follows immediately from \cite[Theorem 1]{AGG1989}.
\end{proof}

\section{Analysis of $\beta$-variations}\label{sec:analysis-beta-variation}
To keep our notation simple we will only treat constant functions $f(x) \equiv f \in (0, 1]$ on $\overline{Q}$,
but it should be clear from the arguments that the extension to general $f = f_0 \ind_D$ 
where $f_0 \in C(\overline{D})$ and $D \subset Q$ is Jordan measurable\footnote{
The core idea of our proof is to partition the domain into sub-cubes
and analyse each piece using tail estimates for critical GMCs, which extend naturally to Jordan measurable sets.
},
or even positive linear combinations of functions of such form, can be achieved by approximations.
Similarly, we assume for convenience that $R = 1$ for the support of the seed kernel $\rho$ in Definition \ref{def:star-scale-invariant},
even though the proof can be easily adapted to treat more general correlation cutoff.

That being said, we still need to introduce some extra notation before discussing any proofs.
For each $n > 1$, define the rescaled field $H_n(y):= X_{\ell(n)}(n^{-1}y)$ where  $\ell(n) := \log n$.
Thanks to the martingale decomposition of $\star$-scale invariant fields, we can write
\begin{align*}
M(Q_{n, \mbf{j}}) 
&= \int_{Q_{\mbf{j}}} W_{\ell(n)}(n^{-1}y) n^{-d} \widetilde{M}^{(\ell(n))}(\dd y)\\
& =: e^{\gamma_c H_n(\mbf{j}) - \left(\frac{\gamma_c^2}{2}+d\right) \ell(n)}\int_{Q_{\mbf{j}}} e^{\gamma_c \left(H_n(y) - H_n(\mbf{j})\right)} \widetilde{M}^{(\ell(n))}(\dd y)
=: e^{\gamma_c H_n(\mbf{j}) - 2d \ell(n)}\widetilde{M}_{n, \mbf{j}}
\end{align*}

\noindent where $\widetilde{M}^{(\ell(n))}$ is obtained from $M^{(\ell(n))}$ in \eqref{eq:coarse-fine-decomposition} by rescaling the dummy integration variable.
In particular, $\widetilde{M}^{(\ell(n))}$ is independent of $\mc{F}_{\ell(n)}$ and has the same distribution as $M$.
We also write
\begin{align}\label{eq:H-notation}
H_{n, \mbf{j}}(x) := H_n(\mbf{j} + x) - H_n(\mbf{j})
\qquad \text{and} \qquad 
H_{n, \mbf{j}}^* := \sup_{x \in Q} |H_{n, \mbf{j}}(x)|.
\end{align}

Unless otherwise specified, we fix $\alpha := 1/\beta \in (0, 1)$
and suppress the dependence of any absolute constants on these parameters;
the following shorthand is also used:
\begin{align} \label{eq:shorthand}
q_{n, \mbf{j}} := \left[\ell(n)^{1/2} f^{\alpha} e^{\gamma_c H_n(\mbf{j}) - 2d \ell(n)}\right]^\beta,
\qquad 
Y_{n, \mbf{j}} 
:= 1- \exp \left\{-q_{n, \mbf{j}}\widetilde{M}_{n, \mbf{j}}^{\beta}\right\}.
\end{align}

The key to proving Theorem \ref{thm:uniform-integrability} and Theorem \ref{thm:critical-to-supercritical}
is to apply Theorem \ref{thm:poisson-approximation} to $Y_{n, \mbf{j}}$ with conditional expectation given $\mc{F}_{\ell(n)}$.
This is a natural choice since the product $\prod_\mbf{j} \left(1 - Y_{n, \mbf{j}}\right)$ will immediately correspond to $\exp\left(- \widetilde{V}_n^{\beta}(f)\right)$ which we want to study.
Note that since $\mr{supp}(\rho) \subset B(0, 1/2)$,
the dependency neighbourhood of $Y_{n, \mbf{j}}$ is simply given by 
\begin{align*}
\mc{B}_{n, \mbf{j}}
:= \{\mbf{k} \in \mc{J}_n: \|\mbf{k} - \mbf{j}\|_\infty \le 1\}
\end{align*}

\noindent and we have $|\mc{B}_{n, \mbf{j}}|\le 3^d$ for every $\mbf{j} \in \mc{J}_n$.

Our analysis will be performed on events with high probability (see Lemma \ref{lem:bound-acosta} for justification):
for each fixed $a \in (0, \tfrac{3}{2})$, $L > 0$ and $\mbf{j} \in \mc{J}_{n}$ let
\begin{align}\label{eq:good-events}
\begin{split}
m_n(a, L) &:= \gamma_c \ell(n) - \frac{a \log \ell(n)}{\gamma_c} + L,\\
A_{n}(a, L) & := \left\{|X_{\ell(n)}(x)| \le m_n(a, L)\quad \text{for all   $x \in [-2, 2]^d$} \right\}\\
\text{and} \qquad A_{n, \mbf{j}}(a, L) & := \left\{|X_{\ell(n)}(n^{-1}\mbf{j})| \le m_n(a, L) \right\}.
\end{split}
\end{align}

\subsection{Estimate for $b_1$}
Let
\begin{align*}
b_{n, 1} := \sum_{\mbf{j} \in \mc{J}_n} \sum_{\mbf{k} \in \mc{B}_{n,\mbf{j}}} \mb{E} \left[Y_{n, \mbf{j}} | \mc{F}_{\ell(n)}\right]\mb{E} \left[Y_{n, \mbf{k}} | \mc{F}_{\ell(n)}\right].
\end{align*}

\noindent Our goal is to establish the following estimate:
\begin{lem}\label{lem:b1-estimate}
For each $a \in (0, 3/2)$ and $L > 0$, there exists some constant $C = C(a, L) \in (0, \infty)$ such that
\begin{align*}
\mb{E}\left[ b_{n, 1} \ind_{A_{n}(a, L)}\right] \le C f^{2\alpha} \ell(n)^{1/2 - a}
\end{align*}

\noindent uniformly in $n \ge 2$.
\end{lem}

\begin{proof}
Recall that $\widetilde{M}^{(\ell(n))}$ is independent of $\mc{F}_{\ell(n)}$ and has the same distribution as $M$.
By Theorem \ref{thm:critical-tail}, for any ($\mc{F}_{\ell(n)}$-measurable) non-negative $\varphi \in C(\overline{Q_{\mbf{j}}})$, we have
\begin{align*}
\mb{P}\left(\widetilde{M}^{(\ell(n))}(\varphi) > u \big| \mc{F}_{\ell(n)}\right) \sim \frac{\int_{Q_{\mbf{j}}} \varphi(y) \dd y}{u \sqrt{\pi d}}
\qquad \text{as $u \to \infty$}.
\end{align*}

\noindent Moreover, we also have the uniform estimate
\begin{align*}
\mb{P}\left(\widetilde{M}^{(\ell(n))}(Q_{\mbf{j}}) > u \big| \mc{F}_{\ell(n)}\right)
= \mb{P}\left(M(Q) > u\right) \le \frac{C_M}{u} \qquad \forall u > 0
\end{align*}

\noindent for some absolute $C_M \in (0, \infty)$,
and hence
\begin{align}\label{eq:tail-uniform-bound}
\mb{P}\left(\widetilde{M}_{n, \mbf{j}} > u \big| \mc{F}_{\ell(n)}\right) \le \frac{C_M e^{\gamma_c H_{n, \mbf{j}}^*}}{u} \qquad \forall u > 0
\end{align}

\noindent where $H_{n, \mbf{j}}^*$ is defined in \eqref{eq:H-notation}.
Thus it follows from \eqref{eq:tail-uniform-bound} and Lemma \ref{lem:laplace-estimate} that
\begin{align}\label{eq:conditional-mean-Y}
\begin{split}
\mb{E} \left[Y_{n, \mbf{k}} | \mc{F}_{\ell(n)}\right]
& \le C_M \Gamma(1-\alpha) f^{\alpha}\ell(n)^{1/2} e^{\gamma_c H_n(\mbf{j}) - 2d\ell(n) +2\gamma_c \sup_{|x| \le 2d}  |H_{n, \mbf{j}}(x)|}
\end{split}
\end{align}

\noindent for all $\mbf{k} \in \mc{B}_{n, \mbf{j}}$. 
Substituting this back into the definition of $b_{n, 1}$, we see that 
\begin{align}
\notag 
&\mb{E}\left[ b_{n, 1} \ind_{A_n(a, L)}\right]
\le \sum_{\mbf{j} \in \mc{J}_n} \sum_{\mbf{k} \in \mc{B}_{n, \mbf{j}}} 
\mb{E} \left[\mb{E} \left[Y_{n, \mbf{j}} | \mc{F}_{\ell(n)}\right]\mb{E} \left[Y_{n, \mbf{k}} | \mc{F}_{\ell(n)}\right] \ind_{A_n(a, L)}\right]\\
\notag
& \le C f^{2\alpha} 3^d \ell(n) \sum_{\mbf{j} \in \mc{J}_n} 
\mb{E}\left[e^{2\gamma_c H_n(\mbf{j}) - 4d\ell(n) + 4\gamma_c \sup_{|x| \le 2d}  |H_{n, \mbf{j}}(x)|} \ind_{A_{n, \mbf{j}}(a, L)}\right]\\
\label{eq:b1-bound1}
& = C f^{2\alpha} 3^d n^d \ell(n)
\mb{E}\left[e^{2\gamma_c X_{\ell(n)}(\mbf{0}) - 4d\ell(n)} e^{4\gamma_c \sup_{|x| \le 2d/n}  |X_{\ell(n)}(x) -  X_{\ell(n)}(\mbf{0})|} \ind_{A_{n, \mbf{0}}(a, L)}\right]
\end{align}

\noindent where the last step follows from the translation invariance of the field $X_{\ell(n)}$.
By Lemma \ref{lem:lgf-oscillations},
\begin{align*}
&\mb{E}\left[e^{2\gamma_c X_{\ell(n)}(\mbf{0}) - 4d\ell(n)} e^{4\gamma_c \sup_{|x| \le 2d/n}  |X_{\ell(n)}(x) -  X_{\ell(n)}(\mbf{0})|}  \ind_{A_{n, \mbf{0}}(a, L)}\right]\\
& \le 
\mb{E}\left[e^{2\gamma_c X_{\ell(n)}(\mbf{0}) - 4d\ell(n)} 
\mb{E}\left[e^{4\gamma_c \sup_{|x| \le 2d/n}  |X_{\ell(n)}(x) -  X_{\ell(n)}(\mbf{0})|} \bigg| X_{\ell(n)}(0)\right] 
 \ind_{A_{n, \mbf{0}}(a, L)}\right]\\
& \le C \exp\left(C \frac{|m_n(a, L)|}{\ell(n)}\right)
\mb{E}\left[e^{2\gamma_c X_{\ell(n)}(\mbf{0}) - 4d\ell(n)}  \ind_{\{|X_{\ell(n)}(\mbf{0})| \le m_n(a, L)\}}\right].
\end{align*}

\noindent Noting that $(2\gamma_c)^2 / 2 = 4d$, the Cameron-Martin theorem implies that the above expression is
\begin{align*}
&\lesssim \mb{P} \left(X_{\ell(n)}(\mbf{0}) \le m_n(a, L) - 2\gamma_c \ell(n)\right)\\
& \le \frac{1}{\sqrt{\ell(n)}} \exp\left\{-\frac{1}{2\ell(n)}\left(\gamma_c \ell(n) + \frac{a \log \ell(n)}{\gamma_c} - L  \right)^2\right\}
\lesssim_{a, L} n^{-d} \ell(n)^{-(1/2+a)},
\end{align*}

\noindent and we conclude that \eqref{eq:b1-bound1} satisfies the desired bound in our Lemma.

\end{proof}

\subsection{Estimate for $b_2$}
We now show that
\begin{align*}
b_{n, 2} := \sum_{\mbf{j} \in \mc{J}_n} \sum_{\mbf{j} \ne \mbf{k} \in \mc{B}_{n,\mbf{j}}} 
\mb{E} \left[Y_{n, \mbf{j}} Y_{n, \mbf{k}} | \mc{F}_{\ell(n)}\right]
\end{align*}

\noindent satisfies the following conditional estimate.
\begin{lem}\label{lem:b2-estimate}
For each $a \in (0, 3/2)$, $L > 0$ and $h \in (\tfrac{1}{2}, \tfrac{1}{2} + \tfrac{1}{2\sqrt{d}})$, 
there exists some constant $C = C(a, L, h) \in (0, \infty)$ such that
\begin{align*}
\mb{E}\left[ b_{n, 2} \ind_{A_{n}(a, L)}\right] \le C f^{2h \alpha} \ell(n)^{(2h-1)(\tfrac{1}{2} -a)}
\end{align*}

\noindent uniformly in $n \ge 2$.
\end{lem}

\begin{proof}
Since $h \alpha < 1$ and $\min(1, u) \le u^{h\alpha}$ for $u \ge 0$,
we have for any $\mbf{j} \ne \mbf{k} \in \mc{B}_{n, \mbf{j}}$ that
\begin{align*}
\mb{E} \left[Y_{n, \mbf{j}} Y_{n, \mbf{k}} | \mc{F}_{\ell(n)}\right]
& \le \mb{E}\left[\min\left(1, q_{n, \mbf{j}} q_{n, \mbf{k}} \widetilde{M}_{n, \mbf{j}}^{\beta}\widetilde{M}_{n, \mbf{k}}^{\beta}\right)  \big|\mc{F}_{\ell(n)}\right]\\
& \le \mb{E}[ \left(M(Q_\mbf{j})M(Q_\mbf{k})\right)^h] e^{\gamma_ch \left(H_{n, \mbf{j}}^*+ H_{n, \mbf{k}}^*\right)} 
\left(q_{n, \mbf{j}} q_{n, \mbf{k}}\right)^{h \alpha}\\
& \le C \ell(n)^h f^{2h\alpha} e^{2\gamma_c h H_n(\mbf{j}) - 4dh \ell(n)}
e^{4 \gamma_c h \sup_{|x| \le 2d} |H_{n, \mbf{j}}(x)|}
\end{align*}

\noindent where the third line is obtained from Lemma \ref{lem:mom-hyperplane}
and the constant $C \in (0, \infty)$ is independent of $n$ and $\mbf{j} \ne \mbf{k} \in \mc{B}_{n, \mbf{j}}$.
Following the same argument as in the proof of the previous lemma, we have 
\begin{align*}
\mb{E}\left[ b_{n, 2} \ind_{A_{n}(a, L)}\right] 
& \lesssim f^{2h\alpha}\ell(n)^h n^{-4dh + \frac{(2\gamma_c h)^2}{2} + d} 
\mb{E}\left[e^{2\gamma_c h H_n(\mbf{0}) - \frac{(2\gamma_c h)^2}{2}\ell(n)}   \ind_{A_{n, \mbf{0}}(a, L)}\right]\\
& \le f^{2h\alpha} \ell(n)^h n^{-4dh + \frac{(2\gamma_c h)^2}{2} + d} \mb{P}\left(X_{\ell(n)}(0) \le m_n(a, L) - 2\gamma_c h \ell(n)\right)\\
& \lesssim f^{2h\alpha} \ell(n)^{h-1/2 - (2h-1)a} n^{d[-4h(1-h) + 1 - (2h-1)^2]}
\end{align*}

\noindent which leads to our claim.
\end{proof}

\subsection{Proof of Theorem \ref{thm:uniform-integrability}}
\subsubsection{Key cutoff estimate and its consequence}
Define 
\begin{align}\label{eq:poisson-mean}
\lambda_n := \sum_{\mbf{j} \in \mc{J}_n}\mb{E}\left[Y_{n, \mbf{j}}| \mc{F}_{\ell(n)}\right].
\end{align}

Our task here is to provide a suitable upper bound on $\lambda_n$.
To do so we need the following cutoff estimate that controls the size of $H_{n, \mbf{j}}^*$.
\begin{lem}\label{lem:local-fluctuation}
There exists some $c \in (0, \infty)$ such that for any $r \in (0, 1)$, the random variable
\begin{align*}
\mc{E}_n:=\sum_{\mbf{j} \in \mc{J}_n} e^{c \|H_{n, \mbf{j}}\|_{C^{1/2}(Q)}^2} M_{\ell(n)}(Q_{n, \mbf{j}})
\end{align*}

\noindent satisfies $\sup_{n \ge 2} \mb{E} \left[\mc{E}_{n}^r \right] < \infty$.
\end{lem}

The proof of Lemma \ref{lem:local-fluctuation} is postponed to Section \ref{subsubsec:cutoff-proof}.
Here we would like to discuss its implication, starting with:
\begin{cor}\label{cor:local-fluctuation}
Let $L > 0$, and define 
\begin{align}\label{eq:holder-compact-set}
  \mc{K}_L := \left\{\varphi \in C(\overline{Q}):\|\varphi\|_{C^{1/2}(Q)}\le L \right\}.
\end{align}

\noindent Following the notation in Lemma \ref{lem:local-fluctuation},
for any $r \in (0, 1)$ there exists some $C = C(r) \in (0, \infty)$ such that the random variable
\begin{align}\label{eq:marked-error}
\widetilde{\mc{E}}_{n}(L, p) :=\sum_{\mbf{j} \in \mc{J}_n} e^{p \|H_{n, \mbf{j}}\|_{C^{1/2}(Q)}} M_{\ell(n)}(Q_{n, \mbf{j}}) \ind_{\{H_{n, \mbf{j}}|_{\overline{Q}} \,\not \in \mc{K}_L\}}
\end{align}

\noindent satisfies
\begin{align*}
\sup_{n \ge 2} \mb{E}\left[\widetilde{\mc{E}}_{n}(L, p)^r\right] \le C \exp\left(\frac{rp^2}{2c} - \frac{rcL^2}{2}\right)
\end{align*}

\noindent for all $p, L \ge 0$.
\end{cor}

\begin{proof}
Using the elementary inequality
\begin{align*}
e^{pm} \ind_{\{m > L\}} 
\le \exp\left(\frac{p^2}{2c} + \frac{cm^2}{2}\right)\ind_{\{m > L\}}
\le \exp\left(\frac{p^2}{2c} - \frac{cL^2}{2}\right) e^{cm^2} \qquad \forall p, m, L \ge 0,
\end{align*}

\noindent we have
\begin{align*}
\widetilde{\mc{E}}_{n}(L, p)
\le \mc{E}_{n} \exp\left(\frac{p^2}{2c} - \frac{cL^2}{2}\right)
\end{align*}

\noindent and hence the desired moment bound follows from Lemma \ref{lem:local-fluctuation}.
\end{proof}

We are now ready to explain the proof of our Laplace transform estimate.
\begin{proof}[Proof of Theorem \ref{thm:uniform-integrability}]
Let $f \equiv q \in [0, \tfrac{1}{2}]$.
Since $Y_{n, \mbf{j}} \in [0, 1]$, we have
\begin{align*}
1-\exp\left(-q \widetilde{V}_n^\beta(Q)\right)
= 1 - \prod_{\mbf{j} \in \mc{J}_n} (1 - Y_{n, \mbf{j}})
\le \min\left(1, \sum_{\mbf{j} \in \mc{J}_n} Y_{n, \mbf{j}}\right).
\end{align*}

\noindent In particular,
\begin{align*}
\mb{E}\left[1-\exp\left(-q \widetilde{V}_n^\beta(Q)\right)\right]
\le \mb{E}\left[\min\left(1, \sum_{\mbf{j} \in \mc{J}_n} Y_{n, \mbf{j}}\right)\right]
\le \mb{E}\left[\min(1, \lambda_n)\right]
\end{align*}

\noindent where the last step follows from Jensen's inequality conditional on $\mc{F}_{\ell(n)}$.
But using the first inequality in \eqref{eq:conditional-mean-Y}, we see that
\begin{align*}
\lambda_n \le Cf^{\alpha}\sum_{\mbf{j} \in \mc{J}_n} e^{2\gamma_c H_{n, \mbf{j}}^*} M_{\ell(n)}(Q_{n, \mbf{j}})
\le Cf^{\alpha} \left[e^{2\gamma_c}M_{\ell(n)}(Q) + \widetilde{\mc{E}}_n(1, 2\gamma_c)\right].
\end{align*}

\noindent As $r / \alpha \in (0, 1)$, the function $x \mapsto x^{r/\alpha}$ is subadditive for $x \ge 0$ and thus 
\begin{align*}
\mb{E}\left[\min(1, \lambda_n)\right]
\le \mb{E}\left[\lambda_n^{r/\alpha}\right]
\le C^{r/\alpha} f^{r}  
\left\{e^{2\gamma_c r / \alpha}\mb{E}\left[M_{\ell(n)}(Q)^{r/\alpha}\right] 
+ \mb{E}\left[\widetilde{\mc{E}}_n(1, 2\gamma_c)^{r/\alpha}\right]\right\}
= O(f^{r})
\end{align*}

\noindent by Lemma \ref{lem:SH-moment-uniform} and Corollary \ref{cor:local-fluctuation}.
This concludes our proof.
\end{proof}

\subsubsection{Proof of the integrability estimate for $\mc{E}_n$}\label{subsubsec:cutoff-proof}
Let us set up some notation.
For any $0 \le s \le t$, recall $X_{s,t} := X_t - X_s$ and define
\begin{align}\label{eq:X-s-t-j}
X_{s, t, \mbf{j}}(x) := \left(X_t - X_s\right)\left(e^{-t}\left(\mbf{j} + x\right)\right),
\qquad \overline{X}_{s, t, \mbf{j}}(x) := X_{s, t, \mbf{j}}(x) - X_{s, t, \mbf{j}}(0).
\end{align}

\noindent One can easily check that
\begin{align}\label{eq:cov-s-t-j}
\mb{E}\left[X_{s, t, \mbf{j}}(x) X_{s, t, \mbf{j}}(y)\right]
= \int_0^{t-s} \rho(e^{-u} (x-y)) du
\end{align}

\noindent and we claim that:

\begin{lem}\label{lem:fluctuation-contract}
There exists $c \in (0, \infty)$ and $C \in (0, \infty)$ independent of $0 \le s \le t$ such that 
\begin{align*}
\sup_{x_0 \in Q, \, \mbf{j} \in \mb{R}^d} 
\mb{E}\left[e^{\gamma_c X_{s, t, \mbf{j}}(x_0) - d(t-s)} \exp\left(c \|\overline{X}_{s, t, \mbf{j}}\|_{C^{1/2}(Q)}^2\right)\right] \le C.
\end{align*}
\end{lem}

\begin{proof}
Let us commence with
\begin{align*}
\mb{E}\left[e^{\gamma_c X_{s, t, \mbf{j}}(x_0) - d(t-s)} \exp\left(c \|\overline{X}_{s, t, \mbf{j}}\|_{C^{1/2}(Q)}^2\right)\right]
& = \mb{E}\left[\exp\left(c \|\widetilde{X}_{s, t, \mbf{j}}\|_{C^{1/2}(Q)}^2\right)\right]
\end{align*}

\noindent where
\begin{align*}
\widetilde{X}_{s, t, \mbf{j}}(x) 
& := \overline{X}_{s, t, \mbf{j}}(x) + \gamma_c \mb{E}\left[ X_{s, t, \mbf{j}}(x_0)\overline{X}_{s, t, \mbf{j}}(x) \right]
=: \overline{X}_{s, t, \mbf{j}}(x) + m_{s, t, \mbf{j}}(x)
\end{align*}

\noindent by the Cameron-Martin theorem.
The shift in mean
\begin{align*}
m_{s, t, \mbf{j}}(x)
& = \gamma_c \int_0^{t-s} \left[\rho(e^{-u}(x_0-x)) - \rho(e^{-u}x_0)\right] du 
= O\left(\gamma_c \left(\|\rho'\|_{\infty} \sup_{x \in Q} |x|\right) \int_0^{t-s} e^{-u}du\right)
\end{align*}

\noindent is uniformly bounded in absolute value and so are its first partial derivatives,
i.e. $\|m_{s, t, \mbf{j}}\|_{C^{1/2}(Q)}$ is bounded uniformly in all parameters. 
Meanwhile, by Theorem \ref{thm:sobolev-inequality} we have
\begin{align*}
\|\overline{X}_{s, t, \mbf{j}}\|_{C^{1/2}(Q)}
\le C_1 \|\nabla \overline{X}_{s, t, \mbf{j}}\|_{L^{p_d}(Q)}
= C_1 \|\nabla X_{s, t, \mbf{j}}\|_{L^{p_d}(Q)}
\end{align*}

\noindent for some absolute constant $C_1 \in (0, \infty)$.
Using \eqref{eq:cov-s-t-j}, we see that
\begin{align*}
\mr{Var}(\partial_1 X_{s, t, \mbf{j}}(x)) 
= -\rho''(0) \int_0^{t-s} e^{-2u} du \le -\rho''(0)
\end{align*}

\noindent is bounded uniformly in all variables, 
and it follows from Lemma \ref{lem:lp-bound-gaussian} that
\begin{align*}
\mb{E}\left[\|\nabla \overline{X}_{s, t, \mbf{j}}\|_{L^{p_d}(Q)}^m\right]^{1/m}
\le C_2 \sqrt{\max(m, p_d)} \qquad \forall m \ge 1
\end{align*}

\noindent for some absolute constant $C_2 \in (0, \infty)$.
Therefore,
\begin{align*}
\mb{E}\left[\exp\left(c \|\widetilde{X}_{s, t, \mbf{j}}\|_{C^{1/2}(Q)}^2\right)\right]
&\le \exp\left(2c \|m_{s, t, \mbf{j}}\|_{C^{1/2}(Q)}^2\right) 
\mb{E}\left[\exp\left(2c C_1^2 \|\nabla X_{s, t, \mbf{j}}\|_{L^{p_d}(Q)}^2\right)\right]\\
&\le \exp\left(2c \|m_{s, t, \mbf{j}}\|_{C^{1/2}(Q)}^2\right) 
\sum_{m \ge 0}\frac{(2c C_1^2)^m}{m!}\mb{E}\left[\|\nabla X_{s, t, \mbf{j}}\|_{L^{p_d}(Q)}^{2m}\right]
\end{align*}

\noindent where the last expression is summable for $c > 0$ sufficiently small,
yielding an upper bound that is uniform in $0 \le s \le t$, $x_0 \in Q$ and $\mbf{j} \in \mb{R}^d$, as claimed.
\end{proof}

\begin{lem}\label{lem:H-n-j-1}
Let $H_{n, \mbf{j}, 1}(x) := X_{\ell(n)/4}\left(n^{-1}(\mbf{j} +x)\right) - X_{\ell(n)/4}\left(n^{-1}\mbf{j}\right)$.
There exists some $c > 0$ and $C \in (0, \infty)$ such that
\begin{align*}
\sup_{n \ge 2} \mb{E}\left[\exp\left(c \sqrt{n} \max_{\mbf{j} \in \mc{J}_n}  \|H_{n, \mbf{j}, 1}\|_{C^{1/2}(Q)}^2\right)\right] < \infty.
\end{align*}
\end{lem}

\begin{proof}
By definition, $H_{n, \mbf{j}, 1}(0) = 0$.
In particular, we have
\begin{align*}
\max_{\mbf{j} \in \mc{J}_n}  \|H_{n, \mbf{j}, 1}\|_{C^{1/2}(Q)}
&\lesssim \max_{\mbf{j} \in \mc{J}_n} \,[H_{n, \mbf{j}, 1}]_{C^{1/2}(Q)}\\
&\lesssim n^{-1/2} \left[X_{\ell(n)/4}\right]_{C^{1/2}(Q)}
\lesssim n^{-1/2} \|\nabla X_{\ell(n)/4}\|_{L^{p_d}(Q)}
\end{align*}

\noindent where the last step follows from Theorem \ref{thm:sobolev-inequality}
and the implicit constants in all the above inequalities are absolute and deterministic.
Since 
\begin{align*}
\mr{Var}(\partial_1 X_{\ell(n)/4}(x)) = -\rho''(0)\int_0^{\ell(n)/4} e^{2u} du
\le C e^{\ell(n)/2} = C \sqrt{n},
\end{align*}

\noindent we obtain from Lemma \ref{lem:lp-bound-gaussian} that
\begin{align*}
\mb{E}\left[ \left(\sqrt{n} \max_{\mbf{j} \in \mc{J}_n}  \|H_{n, \mbf{j}, 1}\|_{C^{1/2}(Q)}^2\right)^m\right]^{1/m}
\lesssim n^{-1/2}\mb{E}\left[  \|\nabla X_{\ell(n)/4}\|_{L^{p_d}(Q)}^{2m}\right]^{1/m}
\lesssim \max(2m, p_d)
\end{align*}

\noindent for all $m \ge 1$ and hence the uniform estimate for exponential moments follows.
\end{proof}

\begin{proof}[Proof of Lemma \ref{lem:local-fluctuation}]
Let us write $H_{n, \mbf{j}}(x) := H_{n, \mbf{j}, 1}(x) + H_{n, \mbf{j}, 2}(x)$ 
where $H_{n, \mbf{j}, 1}$ is defined in Lemma \ref{lem:H-n-j-1}
and $H_{n, \mbf{j}, 2}:= \overline{X}_{s, t, \mbf{j}}$ in \eqref{eq:X-s-t-j} with $t = \ell(n)$ and $s = \ell(n)/4 = t/4$.
Since 
\begin{align*}
\|H_{n, \mbf{j}}\|_{C^{1/2}(Q)}^2
\le 2 \|H_{n, \mbf{j}, 1}\|_{C^{1/2}(Q)}^2 + 2\|H_{n, \mbf{j}, 2}\|_{C^{1/2}(Q)}^2
\end{align*}

\noindent and 
\begin{align*}
M_{\ell(n)}(Q_{n, \mbf{j}})
= \sqrt{\ell(n)}n^{-d}\int_Q 
\left(e^{\gamma_c X_{\ell(n)/4, \ell(n)}(n^{-1}(\mbf{j}+x)) - \frac{3d}{4}\ell(n)}\right) \left(e^{\gamma_c X_{\ell(n)/4}(n^{-1}(\mbf{j}+x)) - \frac{d}{4}\ell(n)}\right)  \dd x,
\end{align*}

\noindent it follows from Lemma \ref{lem:fluctuation-contract} that 
\begin{align*}
\mb{E}\left[\mc{E}_n | \mc{F}_{\ell(n)/4}\right]
\lesssim \sqrt{\ell(n)} n^{-d} \sum_{j \in \mc{J}_n}
\exp\left(2c \|H_{n, \mbf{j}, 1}\|_{C^{1/2}(Q)}^2\right)
\int_Q  \left(e^{\gamma_c X_{\ell(n)/4}(n^{-1}(\mbf{j}+x)) - \frac{d}{4}\ell(n)}\right)  \dd x
\end{align*}

\noindent provided that $c$ is sufficiently small.
Observe that this last expression is 
\begin{align*}
&\lesssim 
\exp\left(2c \max_{\mbf{j} \in \mc{J}_n}\|H_{n, \mbf{j}, 1}\|_{C^{1/2}(Q)}^2\right) 
\sqrt{\ell(n)} n^{-d} \sum_{j \in \mc{J}_n}
\int_Q  \left(e^{\gamma_c X_{\ell(n)/4}(n^{-1}(\mbf{j}+x)) - \frac{d}{4}\ell(n)}\right)  \dd x\\
&= 2\exp\left(2c \max_{\mbf{j} \in \mc{J}_n}\|H_{n, \mbf{j}, 1}\|_{C^{1/2}(Q)}^2\right) M_{\ell(n)/4}(Q)
\end{align*}

\noindent and this implies that for any $r \in (0, 1)$, we have 
\begin{align*}
\mb{E}\left[\mc{E}_n^r\right]
\le \mb{E}\left[\mb{E}\left[\mc{E}_n | \mc{F}_{\ell(n)/4}\right]^r\right]
&\lesssim \mb{E}\left[\left(\exp\left(2c \max_{\mbf{j} \in \mc{J}_n}\|H_{n, \mbf{j}, 1}\|_{C^{1/2}(Q)}^2\right) M_{\ell(n)/4}(Q)\right)^r\right]\\
& \le \mb{E}\left[\exp\left(2c r p \max_{\mbf{j} \in \mc{J}_n}\|H_{n, \mbf{j}, 1}\|_{C^{1/2}(Q)}^2\right)\right]^{1/p} 
\mb{E}\left[M_{\ell(n)/4}(Q)^{\frac{rp}{p-1}}\right]^{\frac{p-1}{p}}
\end{align*}

\noindent for any $p > 1$ by H\"older's inequality.
Choosing $p$ sufficiently large so that $r_p := rp/(p-1) < 1$,
we see that $\mb{E}\left[M_{\ell(n)/4}(Q)^{r_p}\right]^{\frac{p-1}{p}}$ is uniformly bounded in $n$ by Lemma \ref{lem:SH-moment-uniform}.
Since $2rp \le \sqrt{n}$ for $n$ sufficiently large,
we obtain the desired uniform estimate by an application of Lemma \ref{lem:H-n-j-1}.
\end{proof}

\subsection{Proof of Theorem \ref{thm:critical-to-supercritical}}
We now refine the analysis in the previous subsection with the goal of identifying the limit of $\lambda_n$ as $n \to \infty$.
Let us commence with the following observation.
\begin{lem}\label{lem:moment-criticalGMC-compact}
Let $\mc{K} \subset C(\overline{Q})$ be a non-empty compact subset. We have
\begin{align}\label{eq:moment-criticalGMC-compact1}
\sup_{\varphi \in \mc{K}} 
\left|\frac{\mb{E}\left[1 - \exp\left(-q M(e^{\gamma_c \varphi} \ind_Q)^\beta\right)\right]}{q^\alpha \int_Q e^{\gamma_c \varphi(y)}\dd y}
- \frac{\Gamma(1 - \alpha)}{\sqrt{\pi d}} \right| \xrightarrow[q \to 0^+]{} 0.
\end{align}

\noindent In particular, for any $\eps \in (0, 1)$ there exists $q_0 = q_0(\eps) > 0$ such that 
\begin{align}\label{eq:moment-criticalGMC-compact2}
\frac{\mb{E}\left[1 - \exp\left(-q M(e^{\gamma_c \varphi} \ind_Q)^\beta\right)\right]}{q^\alpha \int_Q e^{\gamma_c \varphi(y)}\dd y}
\in \left[\frac{(1-\eps)\Gamma(1 - \alpha)}{\sqrt{\pi d}}, \frac{(1+\eps)\Gamma(1 - \alpha)}{\sqrt{\pi d}}\right]
\end{align}

\noindent for all $\varphi \in \mc{K}$ and $q \in (0, q_0]$.
\end{lem}
\begin{proof}
Let $\delta > 0$ be fixed.
By compactness, we can choose a finite $J = J(\delta)$ number of functions $\varphi_j \in \mc{K}$ such that
$\mc{C}(\varphi_j, \delta):=\{\varphi \in C(\overline{Q}): \|\varphi - \varphi_j\|_\infty < \delta\}_{j \le J}$ forms an open cover of $\mc{K}$.
Then 
\begin{align*}
\sup_{\varphi \in \mc{C}(\varphi_j, \delta)} 
\frac{\mb{E}\left[1 - \exp\left(-q M(e^{\gamma_c \varphi} \ind_Q)^\beta\right)\right]}{q^\alpha \int_Q e^{\gamma_c \varphi(y)}\dd y}
& \le \frac{\mb{E}\left[1 - \exp\left(-q e^{\gamma_c \delta \beta}M(e^{\gamma_c \varphi_j} \ind_Q)^\beta\right)\right]}{q^\alpha e^{-\gamma_c \delta} \int_Q e^{\gamma_c \varphi_j(y)}\dd y}\\
\xrightarrow[q \to 0^+]{} e^{2\gamma_c \delta} \frac{\Gamma(1-\alpha)}{\sqrt{\pi d}}\\
\text{and} \qquad \inf_{\varphi \in \mc{C}(\varphi_j, \delta)} 
\frac{\mb{E}\left[1 - \exp\left(-q M(e^{\gamma_c \varphi} \ind_Q)^\beta\right)\right]}{q^\alpha \int_Q e^{\gamma_c \varphi(y)}\dd y}
& \ge \frac{\mb{E}\left[1 - \exp\left(-q e^{-\gamma_c \delta \beta}M(e^{\gamma_c \varphi_j} \ind_Q)^\beta\right)\right]}{q^\alpha e^{+\gamma_c \delta} \int_Q e^{\gamma_c \varphi_j(y)}\dd y}\\
\xrightarrow[q \to 0^+]{} e^{-2\gamma_c \delta} \frac{\Gamma(1-\alpha)}{\sqrt{\pi d}}
\end{align*}

\noindent by Theorem \ref{thm:critical-tail} and Lemma \ref{lem:laplace-estimate}. Therefore,
\begin{align*}
\limsup_{q \to 0^+} & 
\sup_{\varphi \in \mc{K}} 
\left|\frac{\mb{E}\left[1 - \exp\left(-q M(e^{\gamma_c \varphi} \ind_Q)^\beta\right)\right]}{q^\alpha \int_Q e^{\gamma_c \varphi(y)}\dd y}
- \frac{\Gamma(1 - \alpha)}{\sqrt{\pi d}} \right|\\
& \le \max_{j \le J} \left[ 
\limsup_{q \to 0^+} 
\sup_{\varphi \in \mc{C}(\varphi_j, \delta)} 
\left|\frac{\mb{E}\left[1 - \exp\left(-q M(e^{\gamma_c \varphi} \ind_Q)^\beta\right)\right]}{q^\alpha \int_Q e^{\gamma_c \varphi(y)}\dd y}
- \frac{\Gamma(1 - \alpha)}{\sqrt{\pi d}} \right|
\right]\\
& \le \left(e^{2\gamma_c \delta}-e^{-2\gamma_c \delta}\right) \frac{\Gamma(1 - \alpha)}{\sqrt{\pi d}}.
\end{align*}

\noindent Since $\delta > 0$ is arbitrary, we obtain \eqref{eq:moment-criticalGMC-compact1} by sending $\delta \to 0^+$.
The claim \eqref{eq:moment-criticalGMC-compact2} also follows similarly.
\end{proof}

\begin{proof}[Proof of Theorem \ref{thm:critical-to-supercritical}]
To show stable convergence, 
it suffices to check that \eqref{eq:check-stable} holds for any $\mc{F}_\infty$-measurable random variable $U$ taking values in $[0,1]$.
As mentioned earlier,
we only treat the case where $f(x) \equiv f > 0$ is a non-negative constant here,
even though local approximations would allow one to adapt the arguments below for more general functions.
A further simplification is that one only needs to verify \eqref{eq:check-stable} for $\mc{F}_{t_0}$-measurable random variables for fixed but otherwise arbitrary $t_0 > 0$,
since the general case follows by considering the inequality
\begin{align*}
& \left|\mb{E}\left[U \exp\left(-\widetilde{V}_n^\beta(f)\right)\right]
-\mb{E}\left[U \exp\left(-\frac{\Gamma(1-\alpha)}{\sqrt{\pi d}}\int_Q f(x)^{\alpha} M(\dd x)\right)\right]\right|\\
& \qquad \le 2 \mb{E}|U - U_{t_0}|
+ \left|\mb{E}\left[U_{t_0} \exp\left(-\widetilde{V}_n^\beta(f)\right)\right]
-\mb{E}\left[U_{t_0} \exp\left(-\frac{\Gamma(1-\alpha)}{\sqrt{\pi d}}\int_Q f(x)^{\alpha} M(\dd x)\right)\right]\right|
\end{align*}

\noindent where $U_{t_0} = \mb{E}[U | \mc{F}_{t_0}]$,
and by first taking the limit $n \to \infty$ and then $t_0 \to \infty$ one can obtain the desired conclusion by martingale convergence.

Let us now fix such a $\mc{F}_{t_0}$-random variable $U$,
and suppose $n$ is sufficiently large so that $\ell(n) \ge t_0$.
Then for any $a \in (\tfrac{1}{2}, \tfrac{3}{2})$ and $L_1 > 0$ we have
\begin{align*}
& \left|
\mb{E}\left[U \exp\left(-\widetilde{V}_n^\beta(f)\right)\right]
- \mb{E}\left[U e^{-\lambda_n}\right]
\right|\\
& \qquad \le \mb{E}\left[U \bigg|\mb{E}\left[\exp\left(-\widetilde{V}_n^\beta(f)\right) \bigg|\mc{F}_{\ell(n)}\right] - e^{-\lambda_n}\bigg|\right]\\
& \qquad \le 2\mb{P}(A_n(a, L_1)^c) 
+ \mb{E}\left[\bigg|\mb{E}\left[\exp\left(-\widetilde{V}_n^\beta(f)\right) \bigg|\mc{F}_{\ell(n)}\right] - e^{-\lambda_n}\bigg| \ind_{A_n(a, L_1)}\right]\\
& \qquad \le 2\mb{P}(A_n(a, L_1)^c)  + \mb{E}[b_{n, 1} \ind_{A_n(a, L_1)}]  + \mb{E}[b_{n, 2} \ind_{A_n(a, L_1)}]
\end{align*}

\noindent where the last step follows from Theorem \ref{thm:poisson-approximation}.
This upper bound vanishes as $n \to \infty$ by Lemma \ref{lem:bound-acosta}, Lemma \ref{lem:b1-estimate} and Lemma \ref{lem:b2-estimate}.

To conclude the proof, it suffices to show that 
\begin{align}\label{eq:compare-poisson-cond-mean}
\left|\lambda_n - \frac{\Gamma(1-\alpha)}{\sqrt{\pi d}} M_{\ell(n)}(f^\alpha)\right| \xrightarrow[n \to \infty]{} 0
\qquad \text{in probability}.
\end{align}

\noindent Without loss of generality,
we analyse this absolute difference on the high-probability event $A_{n}(a, L_1)$ on which we have 
\begin{align*}
q_{n, \mbf{j}} \le e^{\gamma_c \beta L_1} f\ell(n)^{(1/2 - a)\beta} \qquad \forall \mbf{j} \in \mc{J}_n
\end{align*}

\noindent which is uniformly small and vanishing as $n \to \infty$.
Let us consider
\begin{align*}
\lambda_n := 
\sum_{\mbf{j} \in \mc{J}_n}\mb{E}\left[Y_{n, \mbf{j}}| \mc{F}_{\ell(n)}\right] \ind_{\{H_{n, \mbf{j}}|_{\overline{Q}} \in \mc{K}_{L_2}\}}
+ \sum_{\mbf{j} \in \mc{J}_n}\mb{E}\left[Y_{n, \mbf{j}}| \mc{F}_{\ell(n)}\right] \ind_{\{H_{n, \mbf{j}}|_{\overline{Q}} \not \in \mc{K}_{L_2}\}}
=: \lambda_{n, 1}(L_2) + \lambda_{n, 2}(L_2)
\end{align*}

\noindent where $H_{n, \mbf{j}}$ is defined in \eqref{eq:H-notation},
$L_2 > 0$ and $\mc{K}_{L_2}$ (defined in \eqref{eq:holder-compact-set}) is a compact subset of $C(\overline{Q})$ by Arzel\`a-Ascoli.
Using the first inequality in \eqref{eq:conditional-mean-Y} again, we see that 
\begin{align*}
\lambda_{n, 2} \le C f^{\alpha} \widetilde{\mc{E}}_n(L_2, 2\gamma_c).
\end{align*}

\noindent where $\widetilde{\mc{E}}_n(L, p)$ is defined in \eqref{eq:marked-error}.
Meanwhile, for any $\eps > 0$ and $n$ sufficiently large,
Lemma \ref{lem:moment-criticalGMC-compact} implies that
\begin{align*}
\frac{\mb{E}\left[Y_{n, \mbf{j}} \big| \mc{F}_{\ell(n)}\right]}{q_{n, \mbf{j}}^\alpha \int_Q e^{\gamma_c H_{n, \mbf{j}}(y)}\dd y}
\in \left[\frac{(1-\eps)\Gamma(1 - \alpha)}{\sqrt{\pi d}}, \frac{(1+\eps)\Gamma(1 - \alpha)}{\sqrt{\pi d}} \right]
\end{align*}

\noindent simultaneously for all $\mbf{j}\in\mc{J}_n$ satisfying $H_{n, \mbf{j}}|_{\overline{Q}} \in \mc{K}_{L_2}$.
In particular, we obtain
\begin{align*}
\lambda_{n, 1} \le  
\frac{(1+\eps)\Gamma(1 - \alpha)}{\sqrt{\pi d}} 
\sum_{\mbf{j}\in\mc{J}_n} f^{\alpha}M_{\ell(n)}(Q_{n, \mbf{j}}) \ind_{\{H_{n, \mbf{j}}|_{\overline{Q}} \in \mc{K}_{L_2}\}}
&\le \frac{(1+\eps)\Gamma(1 - \alpha)}{\sqrt{\pi d}} M_{\ell(n)}(f^{\alpha})
\end{align*}

\noindent and similarly
\begin{align*}
\lambda_{n, 1} 
&\ge \frac{(1-\eps)\Gamma(1 - \alpha)}{\sqrt{\pi d}} \left[M_{\ell(n)}(f^{\alpha}) - f^{\alpha}\widetilde{\mc{E}}_n(L_2, 0)\right].
\end{align*}

Collecting all the estimates for $\lambda_{n, 1}$ and $\lambda_{n,2}$,
the LHS of \eqref{eq:compare-poisson-cond-mean} is upper bounded by
\begin{align*}
\eps\frac{\Gamma(1-\alpha)}{\sqrt{\pi d}} M_{\ell(n)}(f^{\alpha}) 
+ Cf^{\alpha} \left[\widetilde{\mc{E}}_n(L_2, 2\gamma_c) + \widetilde{\mc{E}}_n(L_2, 0)\right]
\end{align*}

\noindent which converges to $0$ in probability
as we send $n \to \infty$ (using the Seneta-Heyde construction of critical chaos $M_{\ell(n)} \to M$ in probability), 
then $L_2 \to \infty$ (by Corollary \ref{cor:local-fluctuation} for the vanishing of $\widetilde{\mc{E}}_n(L_2, p)$) and finally $\eps \to 0^+$.
\end{proof}

\section{The essential spectrum of critical LQG surfaces}\label{sec:essential-spectrum}

\subsection{Optimality of modulus of continuity}
Before we discuss our result about critical LQG surfaces, 
let us first explain the optimal modulus of continuity for critical GMCs below.
\begin{proof}[Proof of Theorem \ref{thm:optimal-mod-cont}]
For simplicity let us take $D = Q = (0, 1)^d$ in the proof below,
even though the general case follows by a similar argument or rescaling the domain.

For any fixed $t \ge 0$, 
the field $X_t(\cdot)$ is almost surely continuous and we have $\sup_{x \in Q} |X_t(x)| < \infty$ by uniform continuity on $\overline{Q}$.
Given the martingale decomposition of our $\star$-scale invariant field and in particular \eqref{eq:coarse-fine-decomposition},
we have 
\begin{align*}
\left(\inf_{y \in Q} W_t(y)\right) M^{(t)}(B) \le M(B) \le \left(\sup_{y \in Q} W_t(y)\right) M^{(t)}(B)
\end{align*}

\noindent for any fixed $t \ge 0$ and cube $B \subset D$.
In particular, we see for
\begin{align*}
    Z_n := \sqrt{n}\sup\left\{M(B): \text{$B \subset [0,1]^d$ and $|B| \in [2^{-d(n+1)}, 2^{-dn}]$}\right\},
\end{align*}

\noindent the event $ E:=\{\limsup_{n \to \infty} Z_n > 0\}$ is in the tail $\sigma$-algebra $\bigcap_{t > 0} \sigma(X_{t, \infty})$,
and thus $\mb{P}(E) \in \{0, 1\}$ by Kolmogorov's zero-one law.

On the other hand, consider
\begin{align}
\notag 
\widetilde{V}_{2^n}^3(Q)
= \sum_{\mbf{j} \in \mc{J}_{2^n}} \left[\sqrt{\log(2^n)}M(Q_{2^n, \mbf{j}})\right]^{3}
&\le \sqrt{\log 2} Z_n \sum_{\mbf{j} \in \mc{J}_{2^n}} \left[\sqrt{\log(2^n)}M(Q_{2^n, \mbf{j}})\right]^{2}\\
\label{eq:pre-slutsky}
&= \sqrt{\log 2} Z_n \widetilde{V}_{2^n}^2(Q).
\end{align}

\noindent By Theorem \ref{thm:critical-to-supercritical},
$\widetilde{V}_{2^n}^{\beta}(Q)$ converges in distribution to a non-trivial limit as $n \to \infty$ for every fixed $\beta > 1$.
Suppose we now assume to the contrary that $\limsup_{n \to \infty} Z_n = 0$ almost surely,
and in particular $Z_n \to 0$ in the sense of convergence in probability.
Then the upper bound in \eqref{eq:pre-slutsky} necessarily converges to $0$ in probability by Slutsky's theorem,
and this would imply $\widetilde{V}^3_{2^n}(Q)$ also vanishes in probability, which is a contradiction.
\end{proof}

The optimal modulus of continuity extends to more general log-correlated Gaussian fields:
\begin{cor}\label{cor:optimal-mod-cont}
Let $G(\cdot)$ be a centred Gaussian field on some bounded open domain $U \subset \mb{R}^d$ with covariance kernel
\begin{align*}
\mb{E}[G(x) G(y)] = -\log|x-y| + g(x,y) \qquad \forall x, y \in U
\end{align*}

\noindent where $g \in H^{s}_{\mr{loc}}(U \times U)$ for some $s > d$.
If $M_G(\dd x)$ is the critical GMC measure associated to $G$, then 
\begin{align}\label{eq:general-modulus}
\mb{P}\left(\limsup_{|B| \to 0} \sqrt{\log \frac{1}{|B|}} M_G(B) > 0\right) = 1
\end{align}

\noindent where the lim sup is taken over cubes $B$ inside any non-empty open $D$ satisfying $\overline{D} \subset U$.
\end{cor}

\begin{proof}
Take any $\star$-scale invariant field $X$ with covariance kernel satisfying Definition \ref{def:star-scale-invariant}, 
and let $M$ be the associated critical GMC measure.
By the regularity assumptions on $g$ and $\rho$,
it is easy to check that
\begin{align*}
g_{\Delta}(x, y):= \mb{E}\left[X(x)X (y)\right] - \mb{E}\left[G(x)G(y)\right] \in H_{\mr{loc}}^{d + \eps}(U \times U) \qquad \text{for some $\eps > 0$}.
\end{align*}

Thus it follows from Theorem \ref{thm:field-decomposition} that
the fields $X$ and $G$ could be coupled in a way such that $X-G$ is H\"older continuous on $\overline{D} \subset U$.
In particular,
\begin{align*}
M_G(B) \ge \exp\left(\min_{x \in \overline{D} } \left[\gamma_c (G-X)(x) + \frac{\gamma_c^2}{2} g_{\Delta}(x,x)\right]\right) M(B)
\end{align*}

\noindent and therefore our desired claim follows from the analogous result for $M$ in Theorem \ref{thm:optimal-mod-cont}.
\end{proof}

\begin{rem}
The lim sup in \eqref{eq:general-modulus} can be taken over all balls in the domain instead,
an equivalent formulation that we shall take for granted below.
\end{rem}

\subsection{Non-compactness of critical LQG resolvent}
We now explain:
\begin{proof}[Proof of Theorem \ref{thm:cLQG-spec}]
Recall the notation in Section \ref{sssec:spectral-lqg} 
(but replace $\mu_{\gamma}$ with the critical LQG measure $M_{\mr{LQG}}$ in the definition of the Dirichlet form constructed in \cite{RV2015}).
It is a standard result in functional analysis (see e.g. \cite[Exercise 4.2]{Dav1995})
that the non-negative self-adjoint operator $-\tfrac{1}{2}\Delta_{\gamma_c}$ has compact resolvent if and only if 
the embedding of the Dirichlet form $\left(\mc{D}, \|\cdot\|_{\mc{D}}\right) \hookrightarrow L^2(M_{\mr{LQG}})$ is compact.
We shall show that the latter condition is false by constructing a sequence of test functions $\varphi_n$ 
that are bounded with respect to the form norm $\| \cdot \|_{\mc{D}}$.

By translating and rescaling the domain $D$,
we may assume without loss of generality that $[-3,3]^2 \subset D$.
Since the Dirichlet Green's function restricted to $(-2,2)^2$ satisfies the regularity condition in Corollary \ref{cor:optimal-mod-cont},
we have
\begin{align}
\mb{P}\left(\limsup_{|B| \to 0} \sqrt{\log \frac{1}{|B|}} M_{\mr{LQG}}(B) > 0\right) = 1
\end{align}

\noindent where the lim sup may be taken over all balls $B$ inside any non-empty bounded open $U \subset Q = (0,1)^2$.
In particular, we have 
\begin{align}\label{eq:LQG-rescaled-blowup}
\sup\left\{ M_{\mr{LQG}}(B) \log \left(|B|^{-1}\right): \text{$B$ is an open ball with $\overline{B} \subset U$}\right\} = \infty \quad a.s.
\end{align}

To proceed, choose a sequence of pairwise disjoint cubes $B_n = \{x \in \mb{R}^2: \|x - x_n\|_\infty < r_n\} \subset Q$:
such a sequence can be constructed by selecting pairwise disjoint sub-cubes at successive levels of the dyadic decomposition of $Q$.
For each $n \ge 1$, let
\begin{align*}
B_{n, 1} := \{x \in \mb{R}^2: \|x - x_n\|_\infty < r_n / 2\},
\qquad B_{n, 2}:= \{x \in \mb{R}^2: \|x - x_n\|_\infty < r_n / 4\}.
\end{align*}

Let $C_n \in (0, \infty)$ be the constant in Lemma \ref{lem:key-capacity-estimate}
with $D_1, D_2, D_3$ corresponding to $B_{n,2}, B_{n, 1}, B_n$ respectively.
Combining this lemma with \eqref{eq:LQG-rescaled-blowup},
we choose for each $n \ge 1$ a sufficiently small ball $B_{n, 3}$ satisfying $\overline{B_{n, 3}} \subset B_{n, 2}$
and a function $\eta_n \in C_c^\infty(B_{n, 1})$ with the properties
\begin{align*}
0 \le \eta_n \le 1,
\qquad \eta_n|_{B_{n, 3}} \equiv 1,
\end{align*}

\noindent such that
\begin{align*}
\frac{1}{2}\int |\nabla \eta_{n}|^2 dx \le \frac{C_n}{\log\left(1+|B_{n,3}|^{-1}\right)}
\qquad \text{and} \qquad 
M_{\mr{LQG}}(B_{n, 3}) \log\left(1+|B_{n, 3}|^{-1}\right) \ge n C_n.
\end{align*}

\noindent If we set $\varphi_n := \eta_n / \|\eta_{n}\|_{L^2(M_{\mr{LQG}})}$,
then  
\begin{align*}
\|\eta_{n}\|_{L^2(M_{\mr{LQG}})}^2 = \int_D \eta_n^2 dM_{\mr{LQG}} \ge M_{\mr{LQG}}(B_{n,3})
\quad \Rightarrow \quad 
\mc{E}(\varphi_n, \varphi_n) 
= \frac{1}{\|\eta_{n}\|_{L^2(M_{\mr{LQG}})}^2} \mc{E}(\eta_n, \eta_n) 
\le \frac{1}{n} 
\end{align*}

\noindent and hence $\|\varphi_n\|_{\mc{D}}^2 = \mc{E}(\varphi_n, \varphi_n) + \|\varphi_n\|_{L^2(M_{\mr{LQG}})}^2 \le 2$.
This means that $(\varphi_n)_{n \ge 1}$ is bounded with respect to $\|\cdot\|_{\mc{D}}$ 
and at the same time forms an orthonormal sequence in $L^2(M_{\mr{LQG}})$,
and thus it cannot have any strongly convergent subsequence,
i.e. the embedding $\left(\mc{D}, \|\cdot\|_{\mc{D}}\right) \hookrightarrow L^2(M_{\mr{LQG}})$ is not compact almost surely.

To conclude our proof, 
let us recall the variational perspective  (see e.g. \cite[Theorem 4.5.3]{Dav1995}) and consider
\begin{align*}
\lambda_m
& := \inf_{\substack{\mc{S} \subset \mc{D} \\ \mr{dim}(\mc{S}) = m}} \sup_{\substack{f \in \mc{S}\\ \|f\|_{L^2(M_{\mr{LQG}})} = 1}} \mc{E}(f, f).
\end{align*}

\noindent By \eqref{eq:LQG-rescaled-blowup} and Theorem \ref{thm:equivalence-poincare-capacity},
we see that almost surely, 
for each $n \ge 1$ there exists a sequence of functions $u_{n, k} \in C_c^\infty(B_n)$
such that $\|u_{n, k}\|_{L^2(M_{\mr{LQG}})} = 1$ and $\mc{E}(u_{n, k}, u_{n, k}) \to 0$ as $k \to \infty$.
In particular, $\lambda_m = 0$ for all $m \ge 1$ almost surely,
and we conclude by \cite[Theorem 4.5.2]{Dav1995} that  $0 \in \sigma_{\mr{ess}}(-\tfrac{1}{2}\Delta_{\gamma_c})$ almost surely.
\end{proof}

\section{Quantitative Fourier decay of critical GMCs}\label{sec:Fourier}
The goal of Section \ref{sec:Fourier} is to establish the following key moment estimate.
\begin{lem}\label{lem:fourier-moment-bound}
For any finite non-empty $\Lambda \subset \mb{Z}^d$, set
\begin{align}\label{eq:fourier-moment-bound-variables}
\widehat{M}_{\Lambda}:= \max_{\mbf{n} \in \Lambda}|\widehat{M}(\mbf{n})|,
\qquad  \text{and} \qquad  
N_{\Lambda} := \min_{\mbf{n} \in \Lambda} \log \|\mbf{n}\|_\infty.
\end{align}

\noindent
Then for any $\beta \in (1, 2]$, $\alpha:= 1/\beta \in [\tfrac{1}{2}, 1)$ and $r \in (0, 1)$,
there exists $C = C(\beta, r) \in (0, \infty)$ such that
\begin{align}\label{eq:key-moment-estimate}
\mb{E}\left[\widehat{M}_{\Lambda}^r\right]
\le C \ell(e|\Lambda|)^{r(1-\alpha)}  N_{\Lambda}^{-r/2}
\end{align}

\noindent uniformly for all finite $\Lambda \subset \mb{Z}^d$ with sufficiently large $N_{\Lambda}$.
\end{lem}

Let us postpone the proof of the above result, and instead explain how this implies our quantitative Fourier decay.

\begin{proof}[Proof of Theorem \ref{thm:fourier-decay}, assuming Lemma \ref{lem:fourier-moment-bound}]
For each $m \ge 1$, let $H_m := 2^m$
and $\Lambda_m := \{\mbf{n} \in \mb{Z}^d: 2^{H_m} \le \|\mbf{n}\|_{\infty} < 2^{H_{m+1}}\}$ so that $\ell(e|\Lambda_m|) \lesssim_d H_m$ and $N_{\Lambda_m} \asymp H_m$ as $m \to \infty$.
We also fix some $r \in (0,1)$, 
and for any given $\eps > 0$ choose $\beta > 1$ sufficiently close to $1$ such that $\alpha - 1 + \eps > 0$. 
Then by Markov's inequality, we have 
\begin{align*}
\mb{P}\left( \max_{\mbf{n} \in \Lambda_m} \left(\ell(\|\mbf{n}\|_\infty)^{1/2-\eps} |\widehat{M}(\mbf{n})|\right) > u \right)
&\le u^{-r}H_m^{r(1/2-\eps)}\mb{E}\left[\widehat{M}_{\Lambda_m}^r\right]
\lesssim u^{-r}H_m^{-r(\alpha - 1+\eps)}
\end{align*}

\noindent by Lemma \ref{lem:fourier-moment-bound}, and this upper bound is summable over $m$.
By Borel-Cantelli, we deduce that
\begin{align*}
\mb{P}\left( \max_{\mbf{n} \in \Lambda_m} \left(\ell(\|\mbf{n}\|_\infty)^{1/2-\eps} |\widehat{M}(\mbf{n})|\right) > u \quad \mr{i.o.}\right) = 0 \qquad \forall u > 0
\end{align*}

\noindent and hence $\lim_{m \to \infty} \max_{\mbf{n} \in \Lambda_m}\ell(\|\mbf{n}\|_\infty)^{1/2-\eps} |\widehat{M}(\mbf{n})|= 0$ almost surely.

As for the tightness result, 
simply choose the singleton set $\Lambda := \{\mbf{n}\}$ and apply Markov's inequality with Lemma \ref{lem:fourier-moment-bound} in a similar way to conclude the proof.
\end{proof}
\subsection{Overview of strategy}
We now outline the proof of Lemma \ref{lem:fourier-moment-bound},
which ultimately relies on our Corollary \ref{cor:uniform-integrability}.
The high-level idea is that we split the Fourier integral into many different pieces which we call `local blocks'
and show how the absolute moments of their sum (i.e. the original Fourier coefficient) can be controlled by those of $\beta$-variations,
thanks to a local randomisation trick that introduces independent Fourier phases with a controlled error.

In order to explain this strategy more precisely, we need to introduce a bit more notation.
In the following, we will write $b_k := 2^{-k}$ and $t_k := k \log 2$. 
We commence with a Lipschitz partition of unity:
let $\eta: \mb{R} \to \mb{R}$ be the cutoff function $\eta(t) := \left(1 - |t|\right)_+$,
and for each $k \ge 1$ and $z \in \mb{Z}^d$ define 
\begin{align*}
\psi_{k, z}(x) := \prod_{r = 1}^d \eta \left(\frac{x_r}{b_k} - z_r\right), \qquad x \in \mb{R}^d.
\end{align*}

\noindent It is straightforward to see that these functions satisfy the following properties:
\begin{itemize}
\item $0 \le \psi_{k, z}(x) \le 1$ everywhere and $\sum_{z \in \mb{Z}^d} \psi_{k, z}(x) = 1$.
\item $\mr{supp}(\psi_{k, z}) \subset b_k \left(z + [-1, 1]^d\right)$,
and thus for each $k \ge 1$, there exists a finite set $\mc{I}_k \subset \mb{Z}^d$ 
such that $\psi_{k, z}|_{Q}$ is a non-trivial function if and only if $z \in \mc{I}_k$.
\item There exists some $C \in (0, \infty)$ independent of $k$ and $z$ such that $\|\nabla \psi_{k, z}\|_{\infty} \le C/ b_k$.
\end{itemize}

\noindent Then our Fourier coefficients can be decomposed as
\begin{align*}
\widehat{M}(\mbf{n}) = \sum_{z \in \mc{I}_k} 
\int_{Q}  \psi_{k, z}(x) e^{2\pi i \mbf{n} \cdot x} M(\dd x)
=: \sum_{z \in \mc{I}_k} M_{k, z}^{\psi}(\mbf{n})
\end{align*}

\noindent and we refer to $M_{k, z}^{\psi}(\mbf{n})$ as the local blocks.

Next we introduce the shift operator $\tau_a$: for any measure $\mu$ and test function $\varphi$, we have
\begin{align*}
\left(\tau_a \mu\right)(A) := \mu(a + A),
\qquad \int \varphi(x) \left(\tau_a \mu\right)(\dd x) := \int \varphi(x-a) \mu(\dd x).
\end{align*}

\noindent Recalling $M(\dd x) = W_s(x) M^{(s)}(\dd x)$ for any $s \ge 0$,
we now define the shifted local blocks, for each $u \in \mb{R}^d$, by
\begin{align}
\notag 
M_{k, z, u}^{\psi}(\mbf{n}) 
&= \int_{Q} \psi_{k, z}(x)  e^{2\pi i \mbf{n} \cdot x} W_{t_k}(x) \left(\tau_u M^{(t_k)}\right)(\dd x)
=: e^{-2\pi i \mbf{n} \cdot u}\left[ M_{k, z}^{\psi}(\mbf{n}) +  M^{(t_k)}(\Delta_{k, z, u}e_{\mbf{n}})\right]
\end{align}

\noindent where $e_{\mbf{n}}(x) := \exp\left(2 \pi i \mbf{n} \cdot x\right)$ and
\begin{align}
\Delta_{k, z, u}(x) := 
\ind_{Q}(x-u) \psi_{k, z}(x-u) W_{t_k}(x-u) 
- \ind_{Q}(x) \psi_{k, z}(x) W_{t_k}(x).
\end{align}

To randomise our shift,
consider a collection of i.i.d. $U_z \sim \mr{Uniform}([-h, h]^d)$ random variables that are independent of $\mc{F}_\infty$ 
(and in particular the underlying field $X$; the choice of $h > 0$ will be specified later), and write 
\begin{align*}
\widetilde{M}_{k, z}^{\psi}(\mbf{n}) := M_{k, z, U_z}^{\psi}(\mbf{n})
\qquad \text{and} \qquad E_{k, z}(\mbf{n}) := M^{(t_k)}(\Delta_{k, z, U_z}e_{\mbf{n}}).
\end{align*}

\noindent Note that by considering the support property it is straightforward to check that 
\begin{align}\label{eq:sum-support-psi}
\sum_{z \in \mb{Z}^d} \psi_{k, z}(x-U_z) \le 9^d \qquad \text{if} \qquad \sup_{z \in \mb{Z}^d} \|U_z\|_{\infty} \le b_k
\end{align}

\noindent which is satisfied if $h \le b_k$, a condition that will be assumed or satisfied for the rest of Section \ref{sec:Fourier}.

As in Section \ref{sec:analysis-beta-variation},
we assume for simplicity that $R = 1$ for the support condition on the seed kernel $\rho$ in Definition \ref{def:star-scale-invariant}
even though the adaptation to the general case is straightforward.
The collection of `relevant indices' $\mc{I}_k$ for the partition of unity is now decomposed into disjoint modulo classes
\begin{align*}
\mc{I}_k:= \bigcup_{\mbf{c} \in \{0, 1, 2, 3, 4\}^d} \mc{I}_k(\mbf{c})
\qquad \text{where} \qquad 
\mc{I}_k(\mbf{c}) := \{z \in \mc{I}_k: z_r \equiv c_r ~\text{(mod $5$)} \quad \forall r \le d\}
\end{align*}

\noindent such that for any $\mbf{c} \in \{0, 1, 2, 3, 4\}^d$ and conditional on $\mc{F}_{t_k}$: 
\begin{itemize}
\item $\left(M_{k, z}^{\psi}(\cdot)\right)_{z \in \mc{I}_k(\mbf{c})}$ is a collection of independent random variables
(since any pair correspond to two effective integration regions separated by a distance at least $2 e^{-t_k} > \frac{1}{2} e^{-t_k}$); and 
\item we have
\begin{align*}
\left(\widetilde{M}_{k, z}^{\psi}(\mbf{n})\right)_{z \in \mc{I}_k(\mbf{c}), \mbf{n} \in \mb{Z}^d} 
\overset{(d)}{=} 
\left(M_{k, z}^{\psi}(\mbf{n})\right)_{z \in \mc{I}_k(\mbf{c}), \mbf{n} \in \mb{Z}^d}.
\end{align*}

\noindent by the translation invariance of $M^{(t_k)}$ (since $h \le b_k$ so that the random shift does not break conditional independence).
\end{itemize}

The purpose of these local random shifts is to expose Fourier cancellation without the need to replace the limiting chaos by a regularised approximation:
this is one distinction between our approach and previous works in the literature.
We choose the shift scale $h$ much smaller than the block scale $b_k$ but large enough for the Fourier phases to oscillate.
By performing the analysis on each modulo class separately,
we can relate the shifted local blocks back to the original ones up to boundary errors
and eventually reduce our problem back to the integrability of power variations of the underlying critical GMC measure.

\subsection{Error arising from random shifts}
Let us begin with the following simple deterministic observation.
\begin{lem}\label{lem:error-random-shift}
Suppose $h \in [0, b_k]$.
There exists some deterministic constant $C \in (0, \infty)$ depending only on $d$ 
such that $\sum_{z \in \mc{I}_k} |E_{k, z}(\mbf{n})| \le C \mc{E}_{k, h}$ uniformly in $\mbf{n} \in \mb{Z}^d$ with
\begin{align}\label{eq:error-shift}
\mc{E}_{k, h} 
:= e^{\omega_k(2h)} M(\partial_{h}Q)
+ \left(\frac{h}{b_k} +  e^{\omega_k(2h)} - 1\right) M(Q)
\end{align}

\noindent where 
\begin{align*}
\partial_{h}Q &:= \{x \in [-2, 2]^d: \|x -  \partial Q\|_\infty \le h\},\\
\qquad Q^{+h} &:= \{x \in [-2, 2]^d: \|x - Q\|_{\infty} \le h\},\\
\text{and} \qquad
\omega_k(h) &:= \sup \left\{|\log W_{t_k}(x) - \log W_{t_k}(y)|: x, y \in Q^{+h}, \|x - y\|_\infty \le h  \right\}.
\end{align*}

\end{lem}

\begin{proof}
Let us start by the telescoping bound
\begin{align}\label{eq:Delta-telescope}
\begin{split}
&|\Delta_{k, z,u}(x)|
\le 
\psi_{k, z}(x-u) W_{t_k}(x-u)  |\ind_{Q}(x-u) - \ind_{Q}(x)|\\
& \quad 
+ \ind_{Q}(x) \psi_{k, z}(x - u)|W_{t_k}(x-u) -  W_{t_k}(x)|
+ \ind_{Q}(x)  W_{t_k}(x)  | \psi_{k, z}(x-u) - \psi_{k, z}(x)|.
\end{split}
\end{align}

For the first term on the RHS, we have
\begin{align*}
\sum_{z \in \mb{Z}^d}&
\int_{\mb{R}^d} \psi_{k, z}(x-U_z) W_{t_k}(x-U_z)  |\ind_{Q}(x-U_z) - \ind_{Q}(x)| M^{(t_k)}(\dd x)\\
& \qquad \le e^{\omega_k(2h)} \int_{\partial_h Q} \left[ \sum_{z \in \mb{Z}^d} \psi_{k, z}(x-U_z)\right] M(\dd x)
\le 9^d e^{\omega_k(2h)} M(\partial_h Q)
\end{align*}

\noindent where the last step follows from \eqref{eq:sum-support-psi}.
Similarly, the second term is responsible for the error
\begin{align*}
&\sum_{z \in \mb{Z}^d}
\int_{\mb{R}^d}  \ind_{Q}(x) \psi_{k, z}(x - U_z)|W_{t_k}(x-U_z) -  W_{t_k}(x)| M^{(t_k)}(\dd x)
\le 9^d \left(e^{\omega_k(2h)} - 1\right) M(Q).
\end{align*}

As for the third error term in \eqref{eq:Delta-telescope}, we have
\begin{align}
\notag
\sum_{z \in \mb{Z}^d}
&\int_{\mb{R}^d} \ind_{Q}(x)  W_{t_k}(x)  | \psi_{k, z}(x-U_z) - \psi_{k, z}(x)|M^{(t_k)}(\dd x)\\
\label{eq:Delta-error3}
& \qquad \qquad = \int_{Q}
\left(\sum_{z \in \mb{Z}^d}| \psi_{k, z}(x-U_z) - \psi_{k, z}(x)|\right) M(\dd x).
\end{align}

\noindent For each fixed $x$, the number of $z \in \mb{Z}^d$ such that at least one of $\psi_{k, z}(x-U_z)$ or $\psi_{k, z}(x)$ is non-zero is bounded by $9^d$.
Using the Lipschitz estimate $\sup_x|\psi_{k, z}(x) - \psi_{k, z}(x - u)| \lesssim \|u\|_{\infty} / b_k$,
we see that \eqref{eq:Delta-error3} is $\lesssim M(Q) h / b_k$.
This concludes our proof.
\end{proof}

Our next task is to control the size of $\omega_k(h)$.
\begin{lem}\label{lem:control-coarse-field}
There exists some constant $C \in (0, \infty)$ such that
\begin{align*}
\mb{P}(\omega_k(h) > C\sqrt{hk 2^k}) \le C e^{-k/C} \qquad \forall h \in \left[ 2^{-2(k+1)}, 1\right], ~k \ge 1.
\end{align*}
\end{lem}

\begin{proof}
Let us consider 
\begin{align*}
\omega_k(h) 
\le \widetilde{\omega}_k(h) 
&:= \sup_{x, y \in [-2, 2]^d: \|x - y\|_\infty \le h} |\log W_{t_k}(x) - \log W_{t_k}(y)|\\
& =  \gamma_c \sup_{x, y \in [-2, 2]^d: \|x - y\|_\infty \le h} |X_{t_k}(x) - X_{t_k}(y)|
\end{align*}

\noindent and instead prove the analogous tail probability bound for $\widetilde{\omega}_k(h)$ for $h \ge 2^{-2(k+1)}$.
Based on a calculation similar to \eqref{eq:natural-distance} and the explanation right afterwards,
we see that
\begin{align*}
  \mb{E}[|X_{t_k}(x) - X_{t_k}(y)|^2] \lesssim e^{t_k}|x-y| \lesssim e^{t_k} h \qquad \text{if} \qquad \|x - y\|_\infty \le h.
\end{align*}

\noindent By Dudley's entropy bound \eqref{eq:Dudley-entropy2},
there exists some absolute constant $C_1 \in (0, \infty)$ such that 
\begin{align*}
\mb{E}\left[\widetilde{\omega}_k(h)\right]
&\le C_1 \int_0^{C_1\sqrt{e^{t_k} h}} \sqrt{\log \left(1 + \frac{1}{(r^2 e^{-t_k})^d} \right)}  dr\\
&\le C_1 \left( C_1 \sqrt{e^{t_k} h}  \int_0^{C_1\sqrt{e^{t_k} h}} \log \left(1 + \frac{1}{(r^2 e^{-t_k})^d} \right) dr\right)^{1/2}
= O(\sqrt{h k 2^k})
\end{align*}

\noindent where the second step follows from Cauchy-Schawrz.
By the Borell-TIS inequality  \eqref{eq:Borell-TIS},
there exists some constant $C_2 \in (0, \infty)$ independent of $h$ and $k$ such that
\begin{align*}
\mb{P}\left(\widetilde{\omega}_k(h) \ge \left(C_2 + u\right)\sqrt{h k 2^k}\right) 
\le C_2 \exp\left(-\frac{\left(u\sqrt{h k 2^k}\right)^2}{C_2 2^k h}\right)
\le C_2\exp\left(-\frac{u^2}{C_2}k\right) 
\qquad \forall u > 0
\end{align*}

\noindent which gives the desired estimate.
\end{proof}

\subsection{Proof of key moment estimate}
\begin{proof}[Proof of Lemma \ref{lem:fourier-moment-bound}]
By H\"older's inequality, it suffices to prove the claim for $r \in (0, 1)$ sufficiently close to $1$
so that $1 - d(1-r)^2 > 0$ for application of Corollary \ref{cor:boundary-momentGMC} at the end of the proof.

Recall the notation in \eqref{eq:fourier-moment-bound-variables}, and set
\begin{align*}
k = k_{\Lambda} := \lfloor N_{\Lambda} / 2 \rfloor
\qquad \text{and} \qquad h = h_{\Lambda} := 2^{-N_{\Lambda}}
\end{align*}

\noindent so that $h \asymp b_{k}^2$.

For each modulo class $\mbf{c} \in \{0, \dots, 4\}^d$, 
let  $S_{k, \mbf{c}}(\mbf{n}) := \sum_{z \in \mc{I}_k(\mbf{c})} M_{k, z}^{\psi}(\mbf{n})$.
Then $\widehat{M}(\mbf{n}) = \sum_{\mbf{c} \in \{0, \dots, 4\}^d} S_{k, \mbf{c}}(\mbf{n})$,
and by the concavity of $x \mapsto x^r$ we have 
\begin{align}\label{eq:pf-fourier-moment}
\mb{E}\left[\widehat{M}_{\Lambda}^r\right]
\le \mb{E}\left[\widehat{M}_{\Lambda}^r \ind_{\mc{G}_{\Lambda}^c}\right]
+ \sum_{\mbf{c}\in\{0, \dots, 4\}^d} 
\mb{E}\left[\max_{\mbf{n} \in \Lambda}\big|S_{k, \mbf{c}}(\mbf{n})\big|^r \ind_{\mc{G}_{\Lambda}}\right]
\end{align}

\noindent where $\mc{G}_{\Lambda} := \{\omega_{k}(2h) \le C\sqrt{h k2^{k}}\}$ is a $\mc{F}_{t_k}$-measurable event.
Using H\"older's inequality with $p \in (1, 1/r)$, the first term on the RHS of \eqref{eq:pf-fourier-moment} is bounded by
$\mb{E}\left[M(Q)^{rp}\right]^{1/p} \mb{P}\left(\mc{G}_{\Lambda}^c\right)^{(p-1)/p} = O\left(e^{-N_{\Lambda} / C}\right)$
by Lemma \ref{lem:control-coarse-field} and Lemma \ref{lem:SH-moment-uniform},
and this error term can thus be safely absorbed.

To proceed, we would like to find a bound on the summand in \eqref{eq:pf-fourier-moment} uniformly over the modulo class $\mbf{c} \in \{0, \dots, 4\}^d$.
Let us recall that our i.i.d. random shifts satisfy $U_z \sim \mr{Uniform}([-h, h]^d)$ and
\begin{align*}
m_{\Lambda}(\mbf{n}) := \mb{E}\left[e^{2\pi i \mbf{n} \cdot U_z}\right] = \prod_{\ell = 1}^d \frac{\sin(2\pi h n_\ell)}{2\pi h n_\ell}.
\end{align*}

\noindent Since there exists some $\ell_* = \ell_*(\mbf{n}) \le d$ such that $|n_{\ell_*}| h = \|\mbf{n}\|_{\infty} h \ge 2^{N_{\Lambda}} h = 1$,
we have
\begin{align*}
|m_{\Lambda}(\mbf{n})| \le \frac{|\sin(2 \pi n_{\ell_*} h)|}{|2 \pi n_{\ell_*} h|} \le \frac{1}{2\pi} \qquad \forall \mbf{n} \in \Lambda.
\end{align*}

Let us define $\widetilde{S}_{k, \mbf{c}}(\mbf{n}) := \sum_{z \in \mc{I}_k(\mbf{c})} \widetilde{M}_{k, z}^{\psi}(\mbf{n})$ as well as
\begin{align*}
\widehat{S}_{k, \mbf{c}}(\mbf{n}) &:= \sum_{z \in \mc{I}_k(\mbf{c})} \left[e^{-2\pi i \mbf{n} \cdot U_z} - m_{\Lambda}(\mbf{n})\right]M_{k, z}^{\psi}(\mbf{n}),
\qquad R_{k, \mbf{c}}(\mbf{n}) := \sum_{z \in \mc{I}_k(\mbf{c})} e^{-2\pi i \mbf{n} \cdot U_z}E_{k, z}(\mbf{n})
\end{align*}

\noindent so that  $\widetilde{S}_{k, \mbf{c}}(\mbf{n}) = \widehat{S}_{k, \mbf{c}}(\mbf{n}) + m_{\Lambda}(\mbf{n})S_{k, \mbf{c}}(\mbf{n}) + R_{k, \mbf{c}}(\mbf{n})$.
By the conditional distributional equality,
\begin{align*}
\mb{E}&\left[\max_{\mbf{n} \in \Lambda}\big|S_{k, \mbf{c}}(\mbf{n})\big|^r \ind_{\mc{G}_{\Lambda}}\right]
= \mb{E}\left[\max_{\mbf{n} \in \Lambda}\big|\widetilde{S}_{k, \mbf{c}}(\mbf{n})\big|^r \ind_{\mc{G}_{\Lambda}}\right]\\
& \le \mb{E}\left[\max_{\mbf{n} \in \Lambda}\big|\widehat{S}_{k, \mbf{c}}(\mbf{n})\big|^r \ind_{\mc{G}_{\Lambda}}\right]
+ \mb{E}\left[\max_{\mbf{n} \in \Lambda}\left|R_{k, \mbf{c}}(\mbf{n})\right|^r \ind_{\mc{G}_{\Lambda}}\right]
+ \left(\max_{\mbf{n} \in \Lambda} |m_{\Lambda}(\mbf{n})|\right)^r 
\mb{E}\left[\max_{\mbf{n} \in \Lambda}\big|S_{k, \mbf{c}}(\mbf{n})\big|^r \ind_{\mc{G}_{\Lambda}}\right]\\
& \le \frac{1}{1- (2\pi)^{-r}}
\left\{
\mb{E}\left[\max_{\mbf{n} \in \Lambda}\big|\widehat{S}_{k, \mbf{c}}(\mbf{n})\big|^r \ind_{\mc{G}_{\Lambda}}\right]
+ \mb{E}\left[\max_{\mbf{n} \in \Lambda}\left|R_{k, \mbf{c}}(\mbf{n})\right|^r \ind_{\mc{G}_{\Lambda}}\right]
\right\}.
\end{align*}

\noindent It is straightforward to check that
\begin{align*}
\sup_{\mbf{n} \in \Lambda} |M_{k, z}^{\psi}(\mbf{n})| \le M(Q \cap \mr{supp}(\psi_{k, z}))
\qquad \text{ and }\qquad 
\sum_{z \in \mc{I}_k(\mbf{c})} M(Q \cap \mr{supp}(\psi_{k, z}))^\beta \lesssim_{d, \beta} V_{2^k}^\beta(Q)
\end{align*}

\noindent where the implicit constants are deterministic.
Using Corollary \ref{cor:high-moment-estimate},
\begin{align*}
& \mb{E}\left[\max_{\mbf{n} \in \Lambda}\big|\widehat{S}_{k, \mbf{c}}(\mbf{n})\big|^r \ind_{\mc{G}_{\Lambda}}\right]
\le \mb{E}\left[\mb{E}\left[\max_{\mbf{n} \in \Lambda}\big|\widehat{S}_{k, \mbf{c}}(\mbf{n})\big|^r \bigg| \mc{F}_{\infty}\right]\right]\\
& \quad \lesssim \ell(e|\Lambda|)^{r(1-\alpha)}\mb{E}\left[\left(\sum_{z \in \mc{I}_k(\mbf{c})} M(Q \cap \mr{supp}(\psi_{k, z}))^\beta \right)^{r\alpha}\right]
\lesssim \ell(e|\Lambda|)^{r(1-\alpha)}\mb{E}\left[V_{2^k}^{\beta}(Q)^{r\alpha}\right]
\end{align*}

\noindent for $r \in (0, 1)$, and this is $O\left(\ell(e|\Lambda|)^{r(1-\alpha)}N_{\Lambda}^{-r/2}\right)$ by Corollary \ref{cor:uniform-integrability}.
Meanwhile, 
\begin{align*}
\mb{E}\left[\max_{\mbf{n} \in \Lambda}\left|R_{k, \mbf{c}}(\mbf{n})\right|^r \ind_{\mc{G}_{\Lambda}}\right]
&\lesssim \mb{E}\left[\mc{E}_{k, h}^r \ind_{\mc{G}_{\Lambda}}\right]
\end{align*}

\noindent by Lemma \ref{lem:error-random-shift}, and this is bounded by
\begin{align}\label{eq:last-bound}
\exp\left(Cr\sqrt{hk2^k}\right) \mb{E}\left[M(\partial_h Q)^r\right] 
+ \left[\frac{h}{b_k} + \exp\left(C\sqrt{hk2^k}\right) - 1\right]^r \mb{E}\left[M(Q)^r\right].
\end{align}

\noindent Combining the estimate for $\mb{E}\left[M(\partial_h Q)^r\right]$ in Corollary \ref{cor:boundary-momentGMC}
with the fact that $hk2^k = O(N_{\Lambda}2^{-N_{\Lambda} / 2})$ and $h/b_k = O(2^{-N_{\Lambda} / 2})$,
we see that \eqref{eq:last-bound} is exponentially small in $N_{\Lambda}$.
Since all the bounds are uniform in $\mbf{c}$ and there are only $5^d$ modulo classes,
we can substitute all the estimates back into \eqref{eq:pf-fourier-moment} and conclude our proof.
\end{proof}



\bibliographystyle{abbrv}
\bibliography{ref}

\end{document}